\documentclass[10pt]{amsart}
\usepackage{amsthm}
\usepackage{amssymb}
\usepackage{amsmath}
\usepackage{fullpage}
\usepackage[all]{xy}

\newtheorem{thm}{Theorem}[section]
\newtheorem{lemma}[thm]{Lemma}
\newtheorem{theorem}[thm]{Theorem}

\newtheorem{proposition}[thm]{Proposition}
\newtheorem{corollary}[thm]{Corollary}
\newtheorem*{corollary*}{Corollary}

\theoremstyle{definition}
\newtheorem{definition}[thm]{Definition}
\newtheorem{remark}[thm]{Remark}
\newtheorem{notation}[thm]{Notation}
\newtheorem{example}[thm]{Example}
\numberwithin{equation}{section}
\newcommand{\cG}{{\mathcal G}}
\newcommand{\cU}{{\mathcal U}}

\newcommand{\R}{\mathbb{R}}
\newcommand{\C}{\mathbb{C}}
\newcommand{\Z}{\mathbb{Z}}

\newcommand{\E}{\mathbb{E}}

\newcommand{\inv}{^{-1}}

\newcommand{\toto}{\rightrightarrows}

\DeclareMathOperator{\graff}{graph}
\DeclareMathOperator{\Hom}{Hom}
\newcommand{\pr}{\operatorname{pr}}

\newcommand{\Man} {\textsf{Man}}
\newcommand{\Bi} {\textsf{Bi}}
\newcommand{\St} {\textsf{St}}
\newcommand{\CFG} {\textsf{CFG}}
\newcommand{\Gr} {\textsf{Gpoid}}
\newcommand{\Gp} {\textsf{Gp}}
\newcommand{\HS} {\textsf{HS}}
\newcommand{\Mat} {\textsf{Mat}}
\newcommand{\Vect} {\textsf{Vect}}
\newcommand{\Cat} {\textsf{Cat}}
\newcommand{\triv}{_\textsf{triv}}
\newcommand{\under}[1] {\underline{#1}}
\newcommand{\sfC}{{\textsf C}}
\newcommand{\sfD}{{\textsf D}}
\newcommand{\sfE}{{\textsf E}}

\begin{document}
\title{Orbifolds as stacks?}

\author{  Eugene Lerman}\thanks{Supported in
part by NSF grant DMS-0603892}

\address{Department of Mathematics, University of Illinois, Urbana, IL 61801}

\email{lerman@math.uiuc.edu}

\begin{abstract}
The first goal of this survey paper is to argue that if orbifolds are
groupoids, then the collection of orbifolds and their maps has to be
thought of as a 2-category.  Compare this with the classical
definition of Satake and Thurston of orbifolds as a 1-category of sets
with extra structure and/or with the ``modern'' definition of
orbifolds as proper etale Lie groupoids up to Morita equivalence.

The second goal is to describe two complementary ways of thinking of
orbifolds as a 2-category:
\begin{enumerate}
\item the weak 2-category of foliation Lie groupoids, bibundles and equivariant maps between bibundles  and
\item the strict 2-category of Deligne-Mumford stacks over the category of smooth manifolds.
\end{enumerate}

\end{abstract}
\maketitle

\tableofcontents

\section{Introduction}
Orbifolds are supposed to be generalizations of manifolds.  While
manifolds are modeled by open balls in the Euclidean spaces, orbifolds
are supposed to be modeled by quotients of open balls by linear
actions of finite groups.  Orbifolds were first defined in the 1950's
by Satake \cite{satake1, satake2}.  The original definition had a number of problems.  The
chief problem was the notion of maps of orbifolds: different papers
of Satake had different definitions of maps and it was never clear if
maps could be composed.  Additionally
\begin{enumerate}
\item The group actions  were required to be effective (and there was a
spurious condition on the codimension of the set of singular points).
The requirement of effectiveness created a host of problems: there
were problems in the definition of suborbifolds and of vector (orbi-)bundles over the orbifolds.  A quotient of a manifold by
a proper locally free action of a Lie group was not necessarily an
orbifold by this definition.
\item  There were problems with pullbacks of
vector (orbi-)bundles --- it was not defined for all maps.
\end{enumerate}

Over the years various patches to the definition have been proposed.
See, for example, Chen and Ruan \cite{ChenRuan}, Haefliger\cite{h1,
h2} , Moerdijk \cite{moer}, Moerdijk and Pronk \cite{MoerPronk}.  In
particular Moerdijk's paper on orbifolds as groupoids has been quite
influential among symplectic topologists.  At about the same time the
notion that orbifolds are Deligne-Mumford/geometric stacks over the
category of manifolds started to be mooted.\\

There are two points to this paper.
\begin{enumerate}
\item If one thinks of orbifolds as groupoids then orbifolds have to be
treated as a 2-category: it is not enough to have maps between
groupoids, one also has to have maps between maps.  This point is not
new; I have learned it from \cite{hm}. Unfortunately it has not been
widely  accepted, and it bears repetition.

\item  There are two complementary ways of thinking of orbifolds as a 
2-category.  One way uses bibundles as maps.  The other requires
embedding Lie groupoids into the 2-category of stacks.  Since stacks
and the related mental habits are not familiar to many differential
geometers I thought it would be useful to explain what stacks are.
While there are several such introductions already available
\cite{Metzler, BXu, heinloth}, I feel there is room for one more, especially 
for the one with the emphasis on ``why.''
\end{enumerate}

I will now outline the argument for thinking of orbifolds as a
2-category (the possibly unfamiliar terms are defined in subsequent
sections).  The simplest solution to all of the original problems with
Satake's definition is to start afresh.  We cannot glue together group
actions, but we can glue together action groupoids.  Given an action
of a finite group, the corresponding groupoid is etale and proper.
This leads one to think of a $C^\infty$ orbifold (or, at least, of an
orbifold atlas) as a proper etale Lie groupoid.  The orbit spaces of
such groupoids are Hausdorff, and locally these groupoids look like
actions groupoids for linear actions of finite groups.  Since a
locally free proper action of a Lie group on a manifold should give
rise to an orbifold, limiting oneself to etale groupoids is too
restrictive. A better class of groupoids consists of Lie groupoid
equivalent to proper etale groupoids.  These are known as foliation
groupoids.

If orbifolds are Lie groupoids, what are maps?  Since many geometric
structures (metrics, forms, vector fields etc) are sections of vector
bundles, hence maps, one cannot honestly do differential geometry on
orbifolds without addressing this question first.

Since groupoids are categories, their morphisms are functors.  But our
groupoids are smooth, so we should require that the functors are
smooth too (as maps on objects and arrows).  One quickly discovers
that these morphisms are not enough.  The problem is that there are
many smooth functors that are equivalences of categories and that have
no smooth inverses.  So, at the very least, we need to formally invert
these smooth equivalences.  But groupoids and functors are not just a
category; there are also natural transformations between functors.
This feature is dangerous to ignore for two reasons.  First of all, it
is ``widely known'' that the space of maps between two orbifolds is
some sort of an infinite dimensional orbifold.  So if one takes the
point of view that orbifolds are groupoids, then the space of maps
between two orbifolds should be a groupoid and not just a set.  The
most natural groupoid structure comes from natural transformations
between functors.  There are other ways to give the space of maps
between two orbifolds the structure of a groupoid, but I don't find
these approaches convincing.

The second reason has to do with gluing maps.  Differential geometers
glue maps all the time.  For example, when we integrate a vector field
on a manifold, we know that a flow exists locally by an existence
theorem for a system of ODEs.  We then glue together these local flows
to get a global flow.  However, if we take the category of Lie
groupoids, identify isomorphic functors and then invert the
equivalences (technically speaking we {\em localize at the
equivalences}), the morphisms in the resulting category will not be
uniquely determined by their restrictions to elements of an open
cover.  We will show that {\bf any} localization of the category of
groupoid will have this feature, regardless of how it is constructed! 
See Lemma~\ref{lem:trouble} below.

Having criticized the classical and ``modern'' approaches to
orbifolds, I feel compelled to be constructive.  I will describe two
{\em geometrically} compelling and complementary ways to  localize
Lie groupoids at equivalences as a 2-category.  These are: 
\begin{enumerate}
\item
replace functors by bibundles and natural transformations by
equivariant maps of bibundles or 
\item  embed groupoids into the
2-category of stacks.  
\end{enumerate}


\noindent
{\bf Acknowledgments }
I have benefited from a number of papers on stacks in algebraic and
differential geometry: Metzler \cite{Metzler}, Behrend-Xu \cite{BXu},
Vistoli \cite{vistoli}, Behrend {\em et alii} \cite{BCEFFGK} and
Heinloth \cite{heinloth} to name a few.  Many definitions and
arguments are borrowed from these papers.  There are now several books
on Lie groupoids. I have mostly cribbed from Moerdijk-Mr\v cun
\cite{MoerMrc}.  The paper by Laurent-Gengoux {\em et alii} \cite{LGTX}
 has also been very useful.

I have benefited from conversations with my colleagues.  In particular
I would like to thank Matthew Ando, Anton Malkin, Tom Nevins and
Charles Rezk.

\subsection{Conventions and notation}
We assume that the reader is familiar with the notions of categories,
functors and natural transformations.  Given a category $C$ we denote
its collection of objects by $C_0$; $C_0$ is not necessarily a set.
The reader may pretend that we are working in the framework of Von~
Neumann - Bernays - G\"{o}del (NBG) axioms for set theory. But for all
practical purposes set theoretic questions, such as questions of {\em
size} will be swept under the rug, i.e., ignored.  We denote the class
of morphisms of a category $C$ by $C_1$.  Given two objects $x,y\in
C_0$ we denote the collection of all morphisms from $x$ to $y$ by
$Hom_C (x,y)$ or by $C(x,y)$, depending on what is less cumbersome.

\subsection{A note on 2-categories}
We will informally use the notions of strict
and weak 2-categories.  For formal definitions the reader may wish to
consult Borceux \cite{Borceux}.  Roughly speaking a strict 2-category
$A$ is an ordinary category $A$ that in addition to objects and
morphisms  has morphisms between morphisms, which are usually called
2-morphisms (to distinguish them from ordinary morphisms which are
called 1-morphisms). We will also refer to 1-morphisms as (1-)arrows. The
prototypical example is $\Cat$, the category of categories.  The
objects of $\Cat$ are categories, 1-morphisms (1-arrows) are functors
and 2-morphisms (2-arrows) are natural transformations between
functors.  We write $\alpha :f\Rightarrow g$ and $\xy
(-8,0)*+{\bullet}="1"; (8,0)*+{\bullet}="2"; {\ar@/^1.pc/^{f}
"1";"2"}; {\ar@/_1.pc/_{g} "1";"2"}; {\ar@{=>}^<<<{\scriptstyle
\alpha} (0,2)*{};(0,-2)*{}} ;
\endxy $, when there is a 2-morphism $\alpha$ from a 1-morphism
$f$ to a 1-morphism   $g$.
Natural transformations can be composed in two different ways:
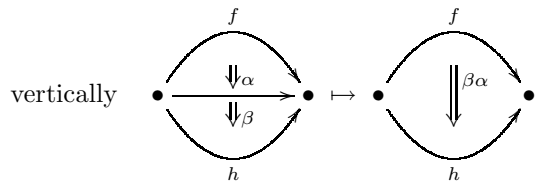
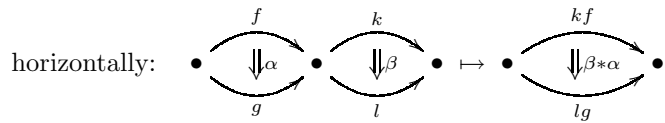
\[
\textrm{vertically} \quad
\xy (-10,0)*+{\bullet}="1"; (10,0)*+{\bullet}="2"; {\ar@/^2.pc/^{f} "1";"2"};
{\ar@/_2.pc/_{h} "1";"2"};{\ar@{->} "1";"2"};
{\ar@{=>}^<<<{\scriptstyle \alpha} (0,4)*{};(0,1)*{}} ;{\ar@{=>}^<<<{\scriptstyle \beta } (0,-1)*{};(0,-4)*{}} ;
\endxy \mapsto
\xy (-10,0)*+{\bullet}="1"; (10,0)*+{\bullet}="2"; {\ar@/^2.pc/^{f} "1";"2"};
{\ar@/_2.pc/_{h} "1";"2"};
{\ar@{=>}^<<<{\scriptstyle \beta \alpha} (0,4)*{};(0,-4)*{}} ; \endxy
\]
and
\[
\textrm{horizontally:} \quad \xy (-8,0)*+{\bullet}="1"; (8,0)*+{\bullet}="2"; {\ar@/^1.pc/^{f} "1";"2"};
{\ar@/_1.pc/_{g} "1";"2"}; {\ar@{=>}^<<<{\scriptstyle \alpha}
(0,2)*{};(0,-2)*{}} ;
 (24,0)*+{\bullet}="3"; {\ar@/^1.pc/^{k} "2";"3"};
{\ar@/_1.pc/_{l} "2";"3"}; {\ar@{=>}^<<<{\scriptstyle \beta }
(16,2)*{};(16,-2)*{}} ;
\endxy
\mapsto
\xy (-10,0)*+{\bullet}="1"; (10,0)*+{\bullet}="2"; {\ar@/^1.pc/^{kf} "1";"2"};
{\ar@/_1.pc/_{lg} "1";"2"};
{\ar@{=>}^<<<{\scriptstyle \beta *\alpha} (-1,2)*{};(-1,-2)*{}} ; \endxy
\]
The two composition are related by a 4-interchange law that we will
not describe.  Axiomatizing this structure gives rise to the notion of a
strict 2-category.

Note that for every 1-arrow $f$ in a 2-category we
have a 2-arrow $id_f: f\Rightarrow f$.  A 2-arrow is invertible if it
is invertible with respect to the vertical composition.  So it makes
sense to talk about two 1-arrows in a 2-category being isomorphic.

Weak 2-categories (also known as bicategories) also have objects,
1-arrows and 2-arrows, but the composition of 1-arrows is no longer
required to be strictly associative.  Rather, given a triple of
composable 1-arrows $f,g, h$ one requires that $(fg)h$ is isomorphic
to $f(gh)$.  That is, one requires that there is an invertible 2-arrow
$\alpha: (fg)h\Rightarrow f(gh)$.  As in a strict 2-category it makes
sense to talk about two 1-arrows in a weak 2-category being isomorphic
(the vertical composition of 2-arrows is still strictly associative).
If $f:x\to y$ is an arrow in a weak 2-category for which there is an
arrow $g: y\to x$ with $fg$ isomorphic to $1_y$ and $gf$ isomorphic to
$1_x$ we say that $f$ is {\em weakly invertible} and that $g$ is a
weak inverse of $f$.

\section{Orbifolds as groupoids}

In this section we define proper etale Lie groupoids.
Proposition~\ref{prop:lin} below is the main justification for
thinking of these groupoids as  orbifolds (or orbifold atlases):
locally they look like finite groups acting linearly on a disk in some
Euclidean space.  Proper etale Lie groupoids are not the only
groupoids we may think of as orbifolds.  For example, a locally free
proper action of a Lie group on a manifold defines a groupoid that is
also, in some sense, an orbifold.  We will explain in what sense such an
action groupoid is equivalent to an etale groupoid.  This requires the
notion of a pullback of a groupoid along a map.
We start by recalling the definition  of a fiber product of sets.
\begin{definition}
 Let $f:X\to Z$ and $g: Y \to Z$ be two maps of sets.  The {\em fiber
 product  } of $f$ and $g$, or more sloppily the fiber product of $X $
 and $Y$ over $Z$ is the set
\[
X \times_{f, Z, g} Y \equiv X\times _Z Y :=
\{ (x, y) \mid f(x) = g(y) \}.
\]
\end{definition}

\begin{remark}
If $f:X\to Z$ and $g:Y \to Z$ are continuous maps between topological
spaces then the fiber product $X\times _Z Y$ is a subset of $X\times
Y$ and hence is naturally a topological space.  If $f:X\to Z$ and $g:Y
\to Z$ are smooth maps between manifolds, then the fiber product is
{\em not} in general a manifold.  It {\em is} a manifold if the map $(f,g):
X\times Y \to Z\times Z$ is transverse to the diagonal $\Delta_Z$.
\end{remark}

\begin{definition}
A {\em groupoid} is a small category (objects form a
set) where all morphisms are invertible.
\end{definition}
  Thus a groupoids $G$
consists of the set of objects (0-morphisms) $G_0$ , the set of arrows
(1-morphisms) $G_1$ together with five structure maps: $s:G_1 \to G_0$
(source), $t:G_1 \to G_0$ (target), $u: G_0 \to G_1$ (unit), $m$
(multiplication) and $inv:G_1 \to G_1$ (inverse) satisfying the
appropriate identities.  We think of an element $\gamma\in G_1$ as an
arrow from its source $x$ to its target $y$:
$x\stackrel{\gamma}{\rightarrow} y$. Thus $s(\gamma) = x$ and
$t(\gamma ) = y$.  For each object $x\in \Gamma_0$ we have the
identity arrow $x\stackrel{1_x}{\longrightarrow} x$, and $u(x) = 1_x$.
Note that $s(u(x)) = t(u(x))= x$.  Arrows with the matching source and
target can be composed: $\xy (-8,0)*+{x \bullet}="x";
(8,0)*+{\bullet y}="y";
(24, 0)*+{\bullet z}="z";
{\ar@/^1.pc/^{\gamma } "x";"y"};
{\ar@/^1.pc/^{\sigma } "y";"z"};
{\ar@/_1.pc/_{\sigma \circ \gamma } "x";"z"};
\endxy$
Therefore the multiplication map $m$  is defined on the fiber product
\[
G_1 \times _{G_0} G_1 \equiv G_1 \times _{s,{G_0},t }G_1
:= \{ (\sigma , \gamma ) \in G_1 \times G_1 \mid s(\sigma) = t (\gamma)\} ;
\]
\[
 m : G_1 \times _{G_0} G_1 \to G_1, \quad m (\xy (-8,0)*+{x \bullet}="x";
(8,0)*+{\bullet y}="y";
(24, 0)*+{\bullet z}="z";
{\ar@/^1.pc/^{\gamma } "x";"y"};
{\ar@/^1.pc/^{\sigma } "y";"z"};
\endxy) = \xy (-8,0)*+{x \bullet}="x";
(8, 0)*+{\bullet z}="z";
{\ar@/^1.pc/^{\sigma \circ \gamma } "x";"z"};
\endxy .
\]
Since all 1-arrows are invertible by assumption ($G$ is a
groupoid) there is the inversion map
\[
inv: G_1 \to G_1, \quad inv
(\xy (-8,0)*+{x \bullet}="x";
(8,0)*+{\bullet y}="y";
{\ar@/^1.pc/^{\gamma } "x";"y"};
\endxy ) =
\xy
(-8,0)*+{x \bullet}="x";
(8,0)*+{\bullet y}="y";
{\ar@/_1.pc/_{\gamma \inv} "y";"x"};
\endxy .
\]
The five maps are subject to identities, some of which we already mentioned.

\begin{notation}
We will write $G = \{ G_1\toto G_0\}$ when we want to emphasize that a groupoid $G$ has the source and target maps.
\end{notation}
\begin{example}
A group is a groupoid with one object.
\end{example}

\begin{example}[sets are groupoids]
Let $M$ be a set, $G_0 = G_1 = M$, $s,t = id:M\to M$, $inv = id$
etc. Then $\{M\toto M\}$ is a groupoid with all the arrows being the
identity arrows.
\end{example}

\begin{example}[action groupoid]
A left action of a group $\Gamma$ on a set $X$ defines an \emph{action
groupoid} as follows: we think of a pair $(g, x)\in \Gamma \times X$
as an arrow from $x$ to $g\cdot x$, where $\Gamma \times X \ni (g, x)
\mapsto g\cdot x \in X$ denotes the action).

Formally $G_1 = \Gamma \times X$, $G_0 = X$, $s(g, x) = x$, $t(g,x) =
g\cdot x$, $u(x) = (1, x)$ where $1\in \Gamma$ is the identity
element, $inv (g,x) = (g\inv, g\cdot x)$ and the multiplication is
given by
\[
(h, g\cdot x) (g, x) = (hg ,x).
\]
\end{example}

\begin{definition}[Orbit space/Coarse moduli space]
Let $G$ be a groupoid.  Then
\[
{\sim} := \{ (x,y) \in G_0 \times G_0
\mid \text{ there is } \gamma \in G_1 \text{ with }
x \stackrel{\gamma}{\longrightarrow } y\}
\]
is an equivalence relation on $G_0$.  We denote the quotient
$G_0/{\sim}$ by $G_0 /G_1$ and think of the projection $G_0 \to
G_0/G_1$ as the orbit map.  We will refer to the set $G_0/G_1$ as the
{\em orbit space } of the groupoid $G$.  Note that if $G = \{ \Gamma
\times X \toto X\} $ is an action groupoid, then $G_0/G_1 = X/\Gamma$.
The orbit space $G_0/G_1$ is also refered to as the {\em coarse moduli
space } of the groupoid $G$.
\end{definition}

\begin{definition}[maps/morphisms of groupoids]
A {\em map/morphism} $\phi$ from a groupoid $G$ to a groupoid $H$ is a
functor.  That is, there is a map $\phi_0 : G_0 \to H_0$ on objects, a
map $\phi_1:G_1 \to H_1$ on arrows that makes the diagram
\[
\xymatrix{
G_1 \ar[r]^{\phi_1} \ar[d]^{(s,t)} & H_1 \ar[d]^{(s,t)} \\
G_0 \times G_0 \ar[r]^{(\phi_0, \phi_0)} & H_0 \times H_0 }
\]
commute and preserves the (partial)
multiplication and the inverse maps.
\end{definition}
\begin{remark}
Note that $\phi_0 = s\circ \phi_1 \circ u$, where $u:G_0\to G_1$
is the unit map.  For this reason we will not distinguish between
a functor $\phi:G\to H$ and the corresponding map on the set of arrows
$\phi_1:G_1 \to H_1$.
\end{remark}

Next we define Lie groupoids. Roughly speaking a \emph{Lie groupoid} is a groupoid in the category
of manifolds.  Thus the spaces of arrows and objects are manifolds and
all the structure maps $s,t, u, m,inv$ are smooth.  Additionally one
usually assumes that the spaces of objects and arrows are is Hausdorff and
paracompact.

There is a small problem with the above definition: in general there
is no reason for the fiber product $G_1 \times _{G_0} G_1$ of a Lie
groupoid $\{ G_1 \toto G_0\}$ to be a manifold.  Therefore one
cannot talk about the multiplication being smooth.  This problem is
corrected by assuming that the source and target maps $s,t:G_1 \to
G_0$ are submersions.  We therefore have:

\begin{definition}
A \emph{Lie groupoid} is a groupoid $G$ such that the set $G_0$ of objects
and the set $G_1$ arrows are (Hausdorff paracompact) manifolds, the source and
target maps $s,t:G_1 \to G_0$ are submersions and all the rest of the
structure maps are smooth as well.
\end{definition}

\begin{remark}
Since $inv^2 = id$, $inv$ is a diffeomorphism.  Since $s\circ inv =
t$, the source map $s$ is a submersion if and only if the target
map $t$ is a submersion.
\end{remark}
\begin{remark}
The coarse moduli space $G_0/G_1$ of a Lie groupoid $G$ is naturally a
topological space.
\end{remark}
\begin{example}[manifolds as Lie groupoids]
Let $M$ be a manifold, $G_0 = G_1 = M$, $s,t = id:M\to M$, $inv = id$
etc. Then $\{M\toto M\}$ is a Lie groupoid with all the arrows being the
identity arrows.
\end{example}

\begin{example}[action Lie groupoid]
Let $\Gamma$ be a Lie group acting smoothly on a manifold $M$.  Then
the action groupoid $\Gamma \times M\toto M$ is a Lie groupoid.
\end{example}

\begin{example}[cover Lie groupoid]
Let $M$ be a manifold with an open cover $\{U_\alpha\}$.  Let $\cU =
\bigsqcup U_\alpha$ be the disjoint union of the sets of the cover
and $\bigsqcup_{\alpha, \beta} U_\alpha \cap U_\beta$ the disjoint
union of double intersections.  More formally
\[
\bigsqcup_{\alpha, \beta} U_\alpha \cap U_\beta = \cU \times _M \cU,
\]
where $\cU= \bigsqcup U_i \to M$ is the evident map.  We define $s:
U_\alpha \cap U_\beta \to U_\alpha$ and $t: U_\alpha \cap U_\beta \to
U_\beta$ to be the inclusions.  Or, more formally, we have two
projection maps $s,t:
\cU \times _M \cU \to \cU$.  We think of a point $x\in U_\alpha \cap
U_\beta$ as an arrow from $x\in U_\alpha$ to  $x\in U_\beta$.
One can check that $   \cU \times _M \cU \toto \cU$ is a Lie
groupoid.  Alternatively it's the pull-back of the groupoid $M\toto
M$ by the ``inclusion'' map $\cU\to M$ (see
Definition~\ref{def-pull-back} below).
\end{example}

\begin{remark}
Occasionally it will be convenient for us to think of {\em a cover of a
manifold $M$ as a surjective local diffeomorphism $\phi:\cU\to M$.}
Here is a justification: If $\{U_i\}$ is an open cover of $M$ then
$\cU = \bigsqcup U_i$ and $\phi:\cU\to M$ is the ``inclusion.''
Conversely, if $\phi:\cU \to M$ is a surjective local diffeomorphism
then there is an open cover $\{V_i\}$ of $\cU$ so that $\phi|_{V_i}: V_i
\to M $ is an open embedding.  Moreover the ``inclusion'' $\bigsqcup
\phi(V_i) \to M$ ``factors'' through $\phi:\cU \to M$. So any cover in
the traditional sense is a cover in the generalized sense.  And any
cover in the new sense gives rise to a cover in the traditional sense.
\end{remark}

\begin{definition}[Proper groupoid]
A Lie groupoid $G$ is {\em proper} if the map $(s,t):G_1 \to
G_0\times G_0$, which sends an arrow to the pair of points (source,
target),  is proper.
\end{definition}

\begin{definition}[Etale groupoid]
A Lie groupoid $G$ is {\em \'etale} if the source and target maps $s,t:G_1 \to G_0$ are local diffeomorphisms.
\end{definition}

\begin{example}
An action groupoid for an action of a finite group is \'etale and
proper.  A cover groupoid $\cU\times _M \cU \toto \cU$ is \'etale and
proper.  An action groupoid $\Gamma \times M \toto M$ is proper if and
only if the action is proper (by definition of a proper action).
\end{example}

\begin{definition}[Restriction of a groupoid]  Let $G$ be a
groupoid and $U\subset G_0$ an open set.  Then $s\inv (U)\cap t\inv
(U)$ is an open submanifold of $G_1$ closed under multiplication and
taking inverses, hence forms the space of arrows of a Lie groupoid
whose  space of objects is $U$.  We call this groupoid {\em the
restriction of $G$ to $U$} and denote it by $G|_U$.
\end{definition}

\begin{remark}
We will see that the restriction is a special case of a pull-back
construction defined below (Definition~\ref{def-pull-back}).
\end{remark}

We can now state the proposition that justifies thinking of proper
etale Lie groupoids as orbifolds.  It asserts
that any such groupoid looks locally like a linear action of a finite
group on an open ball in some $\R^n$.  More precisely, we have:
\begin{proposition} \label{prop:lin}
Let $G$ be a proper etale Lie groupoid.  Then for any
point $x\in G_0$ there is an open neighborhood $U\subset G_0$ so that
the restriction $G|_U$ is isomorphic to an action groupoid $\Lambda
\times U \toto U$ where $\Lambda$ is a finite group.  That is, there
is an invertible functor $f:G|_U \to\{\Lambda \times U \toto U \}$.
Moreover, we may take $U$ to be an open ball in some Euclidean space
centered at the origin and the action of $\Lambda$ to be linear.
\end{proposition}

\begin{proof}
This is a special (easy) case of Theorem~2.3 in \cite{zung}.  For
proper etale {\em effective } groupoids the result was proved earlier in
\cite{MoerPronk}.
\end{proof}

\begin{remark}
One occasionally runs into an idea that a proper etale Lie groupoid
$G$ is an atlas on its coarse moduli space $G_0/G_1$.  Indeed, there is
an analogy with atlases of manifolds: if $M$ is a manifold and
$\{U_i\}$ is a cover by coordinate charts then then $M$ is the coarse
moduli space of the cover groupoid $\{\cU\times _M \cU \toto \cU\}$,
where $\cU = \bigsqcup U_i$.  This idea is leads to a lot of trouble.
\end{remark}

Next I'd like to explain how to obtain a proper etale Lie groupoid
from a proper and locally free action of a Lie group on a manifold.

\begin{definition}\label{def-pull-back}
The {\em pull-back} of a groupoid $G$ by a map $f:N\to G_0$ is the
groupoid $f^*G$ with the space of objects $N$, the space of arrows
\begin{eqnarray*}
(f^*G) _1 &:= & (N\times N) \, \times _{G_0 \times G_0} G_1 \\
&=& \{(x, y , g )\in N\times N \times G_1 \mid s(g) = f(x), t(g) = f(y)\}\\
&=& \{(x, y , g )\in N\times N
\times G_1 \mid f(x) \bullet \xy (-8,0)*+{ }="x";
(8, 0)*+{ }="y";
{\ar@/^1.pc/^{g } "x";"y"};
\endxy \bullet f(y) \},
\end{eqnarray*}
the source and target maps $s(x,y,g) = x$, $t(x,y, g) = y$ and
multiplication given by
\[
(y,z,h) (x, y, g) = (x, z, hg).
\]
Note that  the maps $f_0 = f:N \to G_0$ and $f_1 : f^*G_1
\to G_1$, $f_1 (x, y, g) = g$, form a functor $\tilde{f}: f^*G \to G$.
\end{definition}
It is not always true that the pull-back of a Lie groupoid by a smooth
map is a Lie groupoid: we need the space of arrows $(f^*G)_1$ to be a
manifold and the source and target maps to be submersions. The
following condition turns out to be sufficient.

\begin{proposition} \label{prop1}  Let $G$ be a Lie groupoid and $f:N\to G_0$ a smooth map.
Consider the fiber product
\[
N\times_{f, G_0, s} G_1 = \{ (x,g) \in
N\times G_1 \mid f(x) = s(g)\}.
\]
If the map $N\times_{f, G_0, s} G_1
\to G_0$, $(x,g) \mapsto t(g)$ is a submersion, then the pullback
groupoid $f^*G$ is a Lie groupoid and the functor $\tilde{f}: f^*G
\to G$ defined above is a smooth functor.
\end{proposition}

\begin{proof}
See, for example, \cite{MoerMrc}, pp.\ 121--122.
\end{proof}

\begin{remark}\label{rk:pullback-equiv}
If the map $N\times_{f, G_0, s} G_1
\to G_0$, $(x,g) \mapsto t(g)$ is a {\em surjective} submersion then
the functor $\tilde{f}:f^*G \to G$ is an equivalence of groupoids in
the sense of Definition~\ref{def:equiv-Lie-gpoid} below.
\end{remark}

\begin{example}
Let $G$ be a Lie groupoid, $U$ an open subset of the space of objects
$G_0$.  The inclusion map $\iota:U\hookrightarrow G_0$ satisfies the
conditions of the proposition above and so the pull-back groupoid
$\iota^*G$ is a Lie groupoid.  It is not hard to see that $\iota^*G$ is
the restriction $G|_U$ of $G$ to $U$.
\end{example}

Next recall that an action of a Lie group $\Gamma$ on a manifold $M$
is
\emph{locally free } if for all points $x\in M$ the stabilizer group
\[
\Gamma_x := \{g\in \Gamma \mid g\cdot x = x\}
\]
is discrete.  An action of $\Gamma$ on $M$ is \emph{proper} if the map
\[
\Gamma \times M \to M\times M, \quad (g, x) \mapsto (x, g\cdot x)
\]
is proper (this is exactly the condition for the action groupoid
$\{\Gamma\times M \toto M\}$ to be proper).  A \emph{slice} for an
action of $\Gamma$ on $M$ at a point $x\in M$ is an embedded
submanifold $\Sigma \subset M$ with $x\in\Sigma$ so that
\begin{enumerate}
\item $\Sigma $ is preserved by the action of $\Gamma_x$: for all
$s\in \Sigma$ and $g\in \Gamma_x$, we have $g\cdot s \in \Sigma$.

\item The set $\Gamma \cdot \Sigma := \{g\cdot s \mid g\in \Gamma, s\in
\Sigma\}$ is open in $M$.

\item The map $\Gamma \times \Sigma \to \Gamma \cdot \Sigma \subset
M$, $(g, s) \mapsto g\cdot s$ descends to a diffeomorphism $(\Gamma
\times \Sigma)/\Gamma_x \to \Gamma \cdot \Sigma$ (here $\Gamma _x $
acts on $ \Gamma \times \Sigma$ by $a\cdot (g, s) = (ga\inv , a\cdot
s)$).
\end{enumerate}
Thus, for every point $s\in \Sigma$ the orbit $\Gamma \cdot s$
intersects the slice $\Sigma $ in a unique $\Gamma _x$ orbit.
A classical theorem of Palais asserts that a proper action of a Lie group
$\Gamma$ on a manifold $M$ has a slice at every point of $M$.\\

With these preliminaries out of the way, consider a proper locally
free action of a Lie group $\Gamma$ on a manifold $M$.  Pick a
collection of slices $\{ \Sigma_\alpha\}_{\alpha \in A}$ so that every
$\Gamma$ orbit intersects a point in one of these slices: $\Gamma
\cdot \bigcup \Sigma_\alpha = M$.  Let $\cU = \sqcup \Sigma_\alpha$
and $f:\cU \to M$ be the ``inclusion'' map: for each $x\in \Sigma
_\alpha $, $f(x) = x\in M$.  The fact that $\Sigma_\alpha$'s are
slices implies (perhaps after a moment of thought) that
Proposition~\ref{prop1} applies with $G =
\{\Gamma\times M \toto M\}$ and $f: \cU \to M$. We get a pullback Lie groupoid
$f^*G$, which is, by construction, etale.  By
Remark~\ref{rk:pullback-equiv} the functor $\tilde{f}:f^*\{\Gamma
\times M \toto M\}\to \{\Gamma \times M \toto M\}$ is an equivalence
of groupoids.  Note that $\tilde{f}$ is not surjective and may not be
injective either.  In particular, it's not invertible.  Reasons for
thinking of it as some sort of an isomorphism are explained in the
next section.

Note that if we pull $G$ back further by the inclusion
$\Sigma_\beta \hookrightarrow \sqcup \Sigma_\alpha$, we get an action
groupoid of the form $\Lambda \times \Sigma _\beta \toto \Sigma_\beta$
where $\Lambda $ is  a discrete compact group, that is, a finite group.

\begin{example} An industrious reader may wish to work out the example
of the action of $\C^\times = \{z\in \C\mid z\not =0\}$ on $\C^2
\smallsetminus \{0\}$ given by $\lambda \cdot (z_1, z_2) = (\lambda ^p
z_1, \lambda^q z_2)$ for a pair of positive integers $(p, q)$.  The
reader will only need two slices: $\C \times \{1\}, \{1\}\times \C
\subset \C^2 \smallsetminus
\{0\}$.
\end{example}

\section{Localization and its discontents}

At this point in our discussion of orbifolds we reviewed the reasons
for thinking of smooth orbifolds as Lie groupoids.  If orbifolds are
Lie groupoids then their maps should be smooth functors.  It will turn
out that many such maps that should be invertible are not.  We
therefore need to enlarge our supply of available maps.  We start by
recalling various notions of two categories being ``the same.''  More
precisely recall that there are two equivalent notions of equivalence
of categories.

Recall our notation: if $A$ is a category, then $A_0$ denotes its
collection of objects and $A(a,a')$ denotes the collection of arrows
between two objects $a,a'\in A_0$.
\begin{definition}
A functor $F:A\to B$ is {\em full} if for any $a,a'\in A_0$ the map
$F: A(a,a') \to B(F(a), F(a'))$ is onto.  It is {\em faithful} if $F:
A(a,a') \to B(F(a), F(a'))$ is injective.  A functor that is full and
faithful is {\em fully faithful.}

A functor $F:A\to B$ is {\em essentially surjective} if for any $b\in
B_0$ there is $a\in A_0$ and an invertible arrow $\gamma\in B_1$ from
$F(a)$ to $b$.
\end{definition}

\begin{example}
Let $\Vect$ denote the category of finite dimensional vector spaces
over $\R$ and linear maps. Let $\Mat$ be the category of real
matrices.  That is, the objects of $\Mat$ are non-negative integers.
A morphism from $n$ to $m$ in $\Mat$ is an $n\times m$ real matrix.
The functor $\Mat \to
\Vect$ which sends $n$ to $\R^n$ and a matrix to the corresponding
linear map is fully faithful and essentially surjective.
\end{example}

The following theorem is a basic result in category theory.
\begin{theorem}
A functor $F:A\to B$ is fully faithful and essentially surjective if
and only if there is a functor $G: B\to A$ with two natural
isomorphisms (invertible natural transformations) $\alpha:
FG\Rightarrow id_A$ and $\beta :GF\Rightarrow id_B$.
\end{theorem}

\begin{definition}
A functor $F:A\to B$ satisfying one of the two equivalent conditions
of the theorem above is called an \emph{equivalence of categories}.
We think of the functor $G:B\to A$ above as a (weak) inverse of $F$.
\end{definition}

There is no analogous theorem for $C^\infty$ functors between Lie
groupoids: there are many fully faithful essentially surjective smooth
functors between Lie groupoids with no continuous (weak) inverses.
The simplest examples come from cover groupoids.  If $\cU\times _M \cU
\toto \cU$ is a cover groupoid associated to a cover $\cU \to M$ of
a manifold $M$ then the natural functor $\{\cU\times _M \cU \toto
\cU\} \to \{M\to M\}$ is fully faithful and essentially surjective and
has no continuous weak inverse (unless one of the connected components
of $\cU$ is all of $M$).

Additionally, not every fully faithful and essentially surjective
smooth functor between two Lie groupoids should be considered an
equivalence of Lie groupoids (cf., not every smooth bijection between
manifolds is a diffeomorphism).  The accepted definition is:

\begin{definition}\label{def:equiv-Lie-gpoid}
A smooth functor $F:G\to H$ from a Lie groupoid $G$ to a Lie groupoid
$H$ is an {\em equivalence} of Lie groupoids if
\begin{enumerate}
\item The induced  map
\[
G_1 \to (G_0\times G_0)\times_{(F,F),(H_0,H_0), (s,t)} H_1,
\quad \gamma
\mapsto (s(\gamma), t(\gamma), F(\gamma))
\]
 is a diffeomorphism.
\item  The map
$G_0 \times_{F, H_0, t} H_1 {\to } H_0$,  $(x,h)\mapsto s(h)$ is a surjective
submersion.
\end{enumerate}

\end{definition}
\begin{remark}  The first condition implies that  $F$ is
fully faithful and the second that it is  essentially surjective.
\end{remark}

\begin{remark}
In literature this notion of equivalence variously goes by the names
of ``essential'' and ``weak'' equivalences to distinguish it from
"strict" equivalence: a smooth functor of Lie groupoids $F:G\to H$ is
a {\em strict equivalence} if there is a smooth functor $L: H\to G$ with
two smooth natural isomorphisms (invertible natural transformations)
$\alpha: FL\Rightarrow id_G$ and $\beta :LF\Rightarrow id_H$.  We
will not use the notion of strict equivalence of Lie groupoids in
this paper.
\end{remark}

\begin{example}
As we pointed out above, if $f:\cU\to M$ is surjective local
diffeomorphism then the functor $\tilde{f}: \{\cU\times _M \cU\toto
\cU\}\to \{M\toto M\}$ is an equivalence of Lie groupoids in the sense
of Definition~\ref{def:equiv-Lie-gpoid}.
\end{example}

\begin{example} As we have seen in the previous section,
if a Lie group $\Gamma$ acts locally freely and properly on a manifold
$M$, $\cU =\bigsqcup \Sigma_\alpha$ is a collection of slices with
$\Gamma\cdot \bigcup \Sigma _\alpha = M$ and $f^*\{\Gamma \times M
\toto M\}$ is the pullback of the action groupoid along $f:\cU \to M$
then the functor $\tilde{f}:f^*\{\Gamma \times M \toto M\} \to\{\Gamma
\times M \toto M\}$ is an equivalence  of Lie groupoids.  This is a reason 
for thinking of the action groupoid $\{\Gamma \times M \toto M\}$ as an
orbifold.
\end{example}

\begin{remark}
We cannot fully justify the correctness of
Definition~\ref{def:equiv-Lie-gpoid}.  And indeed the reasons for it
being ``correct'' are somewhat circular.  If one embeds the category
of Lie groupoids either into the the Hilsum-Skandalis category of
groupoids and generalized maps (see below) or into stacks (stacks are
defined in the next section), the functors that become invertible are
precisely the equivalences and nothing else!  But why define the
generalized maps or to embed groupoids into stacks?  To make
equivalences invertible, of course!
\end{remark}

Let us recapitulate where we are.  An orbifold, at this point, should
be a Lie groupoid equivalent to a proper etale Lie groupoid.  If this
is the case, what should be the maps between orbifolds?  Smooth
functors have to be maps in our category of orbifolds, but we need a
more general notion of a map to make equivalences invertible.  There
is a standard construction in category theory called {\em
localization} that allows one to formally invert a class of morphisms.
This is the subject of the next subsection.

\subsection{Localization of a category}

Let $C$ be a category and $W$ a subclass of morphisms of $C$
($W\subset C_1$).  A \emph{localization of $C$ with respect to $W$} is
a category $D$ and a functor $L:C\to D$ with the following properties:
\begin{enumerate}
\item For any $w\in W$, $L(w)$ is invertible in $D$.
\item If $\phi: C \to E$ is a functor with the property that $\phi(w)$ is invertible in $E$ for all $w\in W$ then there exists a unique map $\psi :D\to E$ so that $\psi \circ L = \phi$, that is,
\[
\xy
(0,0)*+{C}="1"; (12,0)*+{D}="2";
(12,12)*+{E} ="3";
{\ar@{->}_{L} "1";"2"};
{\ar@{->}^{\phi} "1";"3"};{\ar@{-->}_{\psi} "2";"3"};
\endxy
\]
commutes
\end{enumerate}
\begin{remark}
The second condition is there to make sure, among other things, that
the localization $D$ is not the trivial category with one object and
one morphism.
\end{remark}

The next two results are old and well known.  The standard reference
is Gabriel-Zisman \cite{GZ}.   We include them for completeness.
\begin{lemma}
If a localization $L:C\to D$ of $C$ with respect to $W \subset C_1$
exists, then it is unique.
\end{lemma}

\begin{proof}
This is a simple consequence of the universal property of the
localization.  If $L':C\to D'$ is another functor satisfying the two
conditions above then there are functors $\psi: D\to D'$ and $\tau:D'
\to D$ so that $\psi \circ L = L'$ and $\tau\circ L' = L$.  Hence
$\tau \circ\psi\circ L = L$.  Since $id_D \circ L = L$ as well, $\tau
\circ \psi = id_D$ by uniqueness.  Similarly $\psi \circ \tau =
id_{D'}$.
\end{proof}

\begin{notation}
We may and will talk about {\bf the} localization of $C$ with respect
to $W$ and denote it by $\pi_W : C \to C[W\inv]$.
\end{notation}

\begin{lemma}\label{lem:exist-local}
The localization $\pi_W : C \to C[W\inv]$ of a category $C$ with
respect to a subclass $W$ of arrows always exists.
\end{lemma}

\begin{remark}
Some readers may be bothered by the issues of {\em size}: the
construction we are about to describe may produce a category where the
collections of arrows between pairs of objects may be too big to be
mere sets.  Later on we will apply Lemma~\ref{lem:exist-local} to the
category of Lie groupoids.  There is a standard solution to this
``problem.''  One applies the argument below only to small categories,
whose collection of objects are sets.  What about the category $\Gr$
of Lie groupoids which is not small (the collection of all Lie
groupoids is a proper class)?  There is a standard solution to this
problem as well.  Fix the disjoint union $\E$ of Euclidean spaces of
all possible finite dimensions; $\E: = \R^0 \sqcup
\R^1\sqcup\ldots \sqcup \R^n\sqcup\ldots $.   Given a Lie groupoid $G$, we 
consider its space of objects $G_0$ as being  embedded in its
space $G_1$ of arrows.  By the Whitney embedding theorem the manifold
$G_1$ may be embedded in some Euclidean space $\R^n \subset \E$.  It
follows that the category $\Gr$ of Lie groupoids is equivalent to the
category of $\E\Gr$ of Lie groupoids embedded in $\E$.  Clearly
$\E\Gr$ is small.
\end{remark}

\begin{proof}[Proof of Lemma~\ref{lem:exist-local}]
The idea of the construction of $C[W\inv]$ is to keep the objects of
$C$ the same, to add to the arrows of $C$ the formal inverses of the
arrows in $W$ and to divide out by the appropriate relations.  Here
are the details.

Recall that a directed graph $\cG$ consists of a class of objects
$\cG_0$, a class of arrows $\cG _1$ and two maps $s, t: \cG_1 \to
\cG_0$ (source and target).
In other words, for us a directed graph is a ``category without
compositions."

Given a category $C$ and a subclass $W\subset C_1$, let $W\inv$ be
the class consisting of formal inverses of elements of $W$: for each
$w\in W$ we have exactly one $w\inv \in W\inv$ and conversely.  We
then have a directed graph $\tilde{C}[W\inv]$ with objects $ C_0$ and
arrows $ C_1 \sqcup W\inv$.

A directed graph $\cG$ generates a free category $F(\cG)$ on $\cG$:
the objects of $F(\cG)$ are objects $\cG_0$ of $\cG$ and arrows are
{\em paths}.  That is, an arrow in
$F(\cG)_1$ from $x\in \cG_0$ to $y\in \cG_0$ is a finite sequence $(\gamma_n,
\gamma_{n-1}, \ldots,\gamma_1)$ of
elements of $\cG_1$ with $s(\gamma_1) = x$ and $t(\gamma_n ) = y$
(think: $x\stackrel{\gamma_1}{\to} \bullet \stackrel{\gamma_2}{\to}
\ldots
\stackrel{\gamma_n}{\to} y$).
In addition, for every $x\in \cG_0$ there is an empty path $(\,)_x $
from $x$ to $x$.  Paths are composed by concatenation:
\[
(\sigma_m,  \ldots,\sigma_1)(\gamma_n,  \ldots,\gamma_1)
= (\sigma_m,  \ldots,\sigma_1,\gamma_n,  \ldots,\gamma_1).
\]

We now construct $C[W\inv]$ from the category $F(\tilde{C}[W\inv])$ by
dividing out the  arrows of $F(\tilde{C}[W\inv])$ by an
equivalence relation.  Namely let $\sim$ be the equivalence relation
generated by the following equations:
\begin{enumerate}
\item $(\,)_x \sim (1_x)$ for all $x\in cG_0$ ($1_x$ is the identity
arrow in $C_1$ for an object $x\in C_0$.
\item $(\sigma) (\gamma) \sim (\sigma \gamma)$ for any pair of composable arrows in $C$.
\item For any $x\stackrel{w}{\to} y\in W$, $(w, w\inv) \sim (1_y) $ and $(w\inv, w) \sim (1_x)$.
\end{enumerate}
Thus we set $C[W\inv]_0 = C_0$ and $C[W\inv]_1 = F(\tilde{C}[W\inv])_1/\sim$.
We have the evident functor $\pi_W:C\to C[W\inv]$ induced by the inclusion of $C$ into the directed graph $\tilde{C}[W\inv]$.

It remains to check that $\pi_W:C\to C[W\inv]$ is a localization.
Note first that for any $w\in W$ the arrow $\pi_W(w)$ is invertible in
$C[W\inv]$ by construction of $C[W\inv]$.  If $\phi:C \to E$ is any
functor such that $\phi(w)$ is invertible for any $w\in W$, then
$\phi$ induced a map $\tilde{\phi}: \tilde{C}[W\inv] \to E$:  $\tilde{\phi} (w\inv) := \phi(w)\inv$.  This map
drops down to a functor $\psi: C[W\inv] \to E$ with $\psi ([w\inv]) =
\phi(w)\inv$ for all $w\in W$ (here $[w\inv]$ denotes the equivalence
class of the path $(w\inv) $ in $F(\tilde{C}[W\inv])$).
\end{proof}

{\bf We now come to a subtle point.} It may be tempting to apply the
localization construction to the category $\Gr$ whose objects are Lie
groupoids, morphisms are functors and the class $W$ consists of
equivalences, and then take the category of orbifolds to be the
subcategory whose objects are isomorphic to proper etale Lie
groupoids.  Let us not rush.  First of all, it will not at all be
clear what the morphisms in $\Gr[W\inv]$ {\em are}, since they are
defined by generators and relations.  A more explicit construction
would be more useful.  Secondly, $\Gr$ is really a 2-category: there
are also natural transformations between functors.  We are thus
confronted with three choices:
\begin{enumerate}
\item Forget about natural transformations and localize; we get a category.
\item Identify isomorphic functors  and then localize.\footnote{
Two smooth functors $f,g:G\to H$ between two Lie groupoids are {\em
isomorphic} if there is a natural transformation $\alpha:G_0 \to H_1$
from $f$ to $g$.  Note that since all arrows in $H_1$ are invertible,
$\alpha$ is automatically a natural isomorphism.} 
We get, perhaps, a smaller category.
\item Localize $\Gr$ as a 2-category.
\end{enumerate}
It is not obvious what the correct choice is.  Option (1) is never
used; perhaps it's not clear how to do it geometrically.  Option (2)
is fairly popular \cite{h4, moer, Hofer}.  There are several
equivalent geometric ways of carrying it out.  We will review the one
that uses isomorphism classes of bibundles.  It is essentially due to
Hilsum and Skandalis \cite{HilsumSkan}.  We will prove that it is, indeed, a
localization.  We will show that it has the unfortunate feature that
maps from one orbifold to another do not form a sheaf: we cannot
reconstruct a map from its restrictions to elements of an open cover.
We will argue that this feature of option (2) is unavoidable: it does
not depend on the way the localization is constructed.  For this
reason I think that choosing option (2) is a mistake.

There is another reason to be worried about option (2). It is ``widely
known'' that the loop space of an orbifold is an orbifold.  So if we
take the point of view that an orbifold is a a groupoid, the loop
space of an orbifold should be a groupoid as well.  But if we think of
the category of manifolds as a 1-category the space of arrows between
two orbifolds is just a set and not a category in any natural sense.
There are, apparently, ways to get around this problem \cite{chen, h4,
LupUribe}, but I don't understand them.

There are many ways of carrying out option (3), localizing
$\Gr$ as a 2-category.  Let me single out three
\begin{itemize}

\item
Pronk constructed a calculus of fractions and localized $\Gr $ as a
weak 2-category \cite{pronk}. She also proved  that the resulting
2-category is equivalent to the strict 2-category of geometric
stacks over manifolds.

\item One can embed the strict 2-category $\Gr$ into a weak
2-category $\Bi$ whose objects are Lie groupoids, 1-arrows are
bibundles and 2-arrows equivariant diffeomorphisms between bibundles.
We will explain the construction of $\Bi$ in the next subsection.

\item One can embed $\Gr$ into the strict 2-category of stacks over manifolds.
We will explain this in section~\ref{sec:stacks}.
\end{itemize}
In the rest of the section we discuss option (2) in details.  We start
by introducing bibundles and reviewing some of their properties.
Thereby we will introduced the weak 2-category $\Bi$. Next we will
discuss a concrete localization of the category of Lie groupoids due
to Hilsum and Skandalis; it amounts to identifying isomorphic 1-arrows
in $\Bi$.  We will then demonstrate that localizing groupoids as
1-category is problematic no matter which particular localization is
being used.

\subsection{Bibundles}\label{bi}

\begin{definition}
A {\em right action} of a groupoid $H$ on a manifold $P$ consists of the
following data:
\begin{enumerate}
\item a map $a:P\to H_0$ (anchor) and
\item a map
\[
P\times _{a, H_0, t} H_1 \to P, \quad (p, h) 
\mapsto p\cdot h, \, \textrm{  (the action)}
\]
(as usual $t:H_1 \to H_0$ denotes the target map) such that
\begin{enumerate}
\item $a(p\cdot h) = s(h)$
\quad\textrm{ for all } $(p,h)\in P\times_{a,H_0,t}H_1$;
\item $(p\cdot h_1) \cdot h_2 = p\cdot (h_1 h_2)$
\quad\textrm{ for all appropriate } $p\in P$ and $h_1, h_2 \in H_1$ ;
\item $p\cdot 1_{a(p)} = p$ for all $p\in P$.
\end{enumerate}
\end{enumerate}

\end{definition}

\begin{definition}
A manifold $P$ with a right action of a groupoid $H$ is a {\bf
principal (right) $H$-bundle over $B$} if there is a surjective
submersion $\pi:P\to B$ so that
\begin{enumerate}
\item $\pi (p\cdot h)  = \pi (p)$ for all
$(p, h) \in P\times _{a, H_0, t} H_1$, that is, $\pi$ is $H$-invariant; and
\item  the map $P\times _{a, H_0, t} H_1 \to P\times _B P$, $(p, h) \mapsto (p,
 p\cdot h)$ is a diffeomorphism, that is, $H$ acts freely and
 transitively on the fibers of $\pi:P\to B$.
\end{enumerate}
\end{definition}
\begin{example}
For a groupoid $H$ the target map $t:H_1 \to H_0$ makes $H_1 $ into
a principal $H$-bundle with the action of $H$ being the
multiplication on the right (the anchor map is $s:H_1 \to H_0$).
This bundle is sometimes called the {\em unit} principal $H$ bundle
for the reasons that may become  clear later.
\end{example}

Principal $H$-bundles pull back: if $\pi: P\to B$
is a principal $H$ bundle and $f:N\to B$ is a map then the pullback
\[
f^*P:= N\times _B P \to N
\]
is a principal $H$ bundle as well.  The action of $H$ on $f^*P$ is the
restriction of the action of $H$ on the product $N\times P$ to
$N\times _B P \subset N\times P$.  It is not difficult to check that
$f^*P \to N$ is indeed a principal $H$-bundle.

\begin{lemma}\label{lem:princH-global}
A principal $H$-bundle $\pi: P\to B$ has a global section if and
only if $P$ is isomorphic to a pull-back of the principal $H$-bundle
$H_1\stackrel{t}{\to } H_0$.
\end{lemma}

\begin{proof}
Since $P\to B$ is $H$-principal we have a diffeomorphism
\[
P\times _{a, H_0, t}H_1 \to P\times _B P, \quad (p, h) \mapsto (p,
p\cdot h).
\]
Its inverse is of the form $(p_1, p_2) \mapsto (p_1, d(p_1, p_2)) \in
P\times _{a, H_0, t}H_1 $, where  $d(p_1, p_2)$ is the unique element
$h$ in $H_1$ so that $p_2 = p_1 \cdot h$.  The map 
\[
d: P\times _B P\to H_1 \quad \textrm{(``the division map'')}
\]
 is smooth.  Note that $d(p,p) = 1 _{a(p)}$.  If $\sigma :B\to P$
is a section of $\pi: P\to B$, define $\tilde{f}:P\to H_1$ by
\[
\tilde{f}(p) = d( \sigma (\pi(p)), p).
\]
Then
\[
p = \sigma (\pi (p)) \cdot \tilde{f} (p)  \quad \textrm{for all }
p\in P.
\]
Note that $\tilde{f}$ is $H$-equivariant: observe that for all $(p,h)
\in P\times _{H_0} H_1$
\[
\sigma(\pi( p\cdot h)) \cdot \tilde{f}(p) \cdot h = p\cdot h =
\sigma (\pi (p\cdot h) ) \tilde{f} (p\cdot h).
\]
Hence, since $P$ is $H$-principal,
\[
\tilde{f}(p) \cdot h = \tilde{f} (p\cdot h).
\]
Consequently we get a map
\[
\varphi:P\to f^*H_1, \quad \varphi(p)= (\pi(p), \tilde{f} (p)),
\]
where $f:B\to H_0$ is defined by  $f(b) = a (\sigma (b))$.   The map $\varphi$ has a smooth inverse $\psi: f^*H_1 \to P$: $\psi (b,h) = \sigma (b)\cdot h$, hence $\varphi$ is a diffeomorphism.

Conversely, since $H_1 \stackrel{t}{\to} H_0$ has a global section,
namely $u(x) = 1_x$ for $x\in H_0$, any pullback of $H_1
\stackrel{t}{\to} H_0$ has a global section as well.
\end{proof}

\begin{remark}
It is useful to think of principal groupoid bundles with global
sections as trivial principal bundles. 
\end{remark}

The next result is technical and won't be needed until we start
discussing stacks in the next section.  It should be skipped on the
first reading.

\begin{corollary}\label{cor:g-equiv-iso}
Let $G$ be a Lie groupoid,  $\xi_1 \to N$, $\xi_2 \to N$ two
principal $G$ bundles with anchor maps $a_1, a_2$ respectively. Any
$G$-equivariant map $\psi: \xi _1 \to \xi_2$ inducing the identity
on $N$ is a diffeomorphism.
\end{corollary}

\begin{proof}
Note that the $a_2 \circ \psi = a_1$; this is necessary for $\psi$
to intertwine the two $G$ actions.

Since $\psi$ is $G$-equivariant and induces the identity map on the
base $N$, for any open set $U\subset N$, $\psi (\xi_1|_U) \subset
\xi_2|_U$.  Therefore it's enough to show that for any sufficiently
small subset $U$ of $N$ the map  $\psi:\xi_1|_U \to \xi_2|_U$ is a
diffeomorphism.   Since $\xi_1 \to N$ is a submersion, it has local
sections. The two observations above allows us to assume that $\xi_1
\to N$ has a global section $\sigma:N\to \xi_1$.

We have seen in the proof of Lemma~\ref{lem:princH-global} that the
section $\sigma$ together with the ``division map" $d: \xi_1 \times
_N \xi_1 \to G_1$ defines a $G$-equivariant diffeomorphism
\[
\tilde{f} :\xi_1 \to f^* (G_1 \to G_0),
\]
where $f = a_1 \circ \sigma$.
 Similarly the section $\psi \circ
\sigma :N\to \xi_2$ together with the division map for $\xi_2$
defines a $G$-equivariant diffeomorphism
\[
\tilde{h}: \xi_2 \to h^* (G_1 \to G_0),
\]
where $h = a_2 \circ (\psi \circ \sigma)$.  Since $(a_2 \circ \psi)
\circ \sigma = a_1 \circ \sigma$, we have $h =f$.  By tracing through
the definitions one sees that
\[
 \psi = (\tilde{h})\inv \circ \tilde{f}.
\]
Hence $\psi$ is a diffeomorphism.
\end{proof}

\begin{definition}
A {\em left} action of a groupoid $G$ on a manifold $M$ is
\begin{enumerate}
\item A map $a_L = a:M\to G_0$ (the  (left) anchor) and
\item a map
\[
 G_1\times_{s, G_0, a} M \to M, \quad (\gamma, x) \mapsto \gamma
\cdot x, \textrm{ (the action) },
\]
such that
\begin{enumerate}
\item $1_{a(x)} \cdot x = x \textrm{ for all } x\in M$,
\item $a(\gamma \cdot x) = t(\gamma)$ for all
$(\gamma, x) \in G_1\times_{s, G_0, a} M $,
\item $\gamma_2\cdot (\gamma _1 \cdot x) = (\gamma_2\gamma_1) \cdot
x$ for all appropriate $\gamma_1, \gamma_2 \in G_1$ and $x\in M$.
\end{enumerate}
\end{enumerate}
\end{definition}

\begin{remark}
Given a right action $a_R: M\to G_0$, $M\times_{G_0} G_1 \to M$ of a
groupoid $G$ on a manifold $M$, we get a left action of $G$ on $M$
by composing it with the inversion map $G_1\to G_1$,  $\gamma \mapsto
\gamma\inv$.
\end{remark}
\begin{remark}\label{rm:bibbundle-from-functor}
If $f:G\to H$ is a smooth functor between two Lie groupoids then the
pullback
\[
f_0^*H_1 =G_0 \times _{f_0, H_0, t} H_1\stackrel{\pi}{\to} G_0
\]
of the principal $H$-bundle $H_1\stackrel{t}{\to} H_0$ by $f_0:G_0
\to H_0$ is  a principal $H$-bundle.  In addition it has a left
$G$-action:
\[
G_1 \times_{s, G_0, \pi} (G_0 \times _{f_0, H_0, t} H_1) \to (G_0
\times _{f_0, H_0, t} H_1), \quad (g, (x,h))\mapsto (t(g), f_1(g)h).
\]
This left $G$-action commutes with the right $H$-action.
\end{remark}
The manifold $f_0^*H_1$ with the commuting actions of $G$ and $H$
constructed above is an example of a {\em bibundle} from $G$ to $H$,
which we presently define.

\begin{definition}
Let $G$ and $H$ be two groupoids.  A {\em bibundle} from $G$ to $H$
is a manifold $P$ together with two maps $a_L:P\to G_0$, $a_R:P\to
H_0$ such that
\begin{enumerate}
\item there is a left action of $G$ on $P$ with respect to an anchor $a_L$ and
a right action of $H$ on $P$ with respect to an anchor $a_R$;
\item $a_L:P \to G_0$ is a principal $H$-bundle
\item $a_R$ is $G$-invariant: $a_R(g\cdot p) = a_R (p)$ for all
$(g,p) \in G_1 \times_{H_0} P$;
\item the actions of $G$ and $H$ commute.
\end{enumerate}
\end{definition}
If $P$ is a bibundle from a groupoid $G$ to a groupoid $H$ we write
$P:G\to H$.

\begin{definition}\label{def:iso-bibund}
Two bibundles $P, Q:G\to H$ are {\em isomorphic} if there is a
diffeomorphism $\alpha: P\to Q$ which is $G-H$ equivariant: $\alpha
(g\cdot p\cdot h) = g\cdot \alpha (p) \cdot h$ for all $(g, p, h) \in
G_1 \times _{G_0} P \times _{H_0}H_1$.
\end{definition}

\begin{remark}[bibundles defined by functors]
By Remark~\ref{rm:bibbundle-from-functor} any functor $f:G\to H$ defines
a bibundle
\[
\langle f\rangle:= f_0^*H_1 = G_0 \times _{f, H_0, t}H_1: G\to H.
\]
The bibundle $\langle id_G\rangle $ corresponding to the identity
functor $id_G:G\to G$ is $G_1$ with $G$ acting on $G_1$ by left and
right multiplications.

Note that $\langle f\rangle \to G_0$ has a global section $\sigma (x)
:= (x, f(1_x))$.
\end{remark}

\begin{example}\label{ex:man-as-groupoid}  A map $f:M\to N$ between
two manifolds  tautologically
defines a functor $f:\{M\toto M\}\to \{N\toto N\}$.  The corresponding
bibundle $\langle f\rangle $ is simply the graph $\graff(f)$ of
$f$. It is not hard to show that a converse is true as well: any
bibundle $P:
\{M\toto M\}\to \{N\toto N\}$ is a graph of a function $f_P: M\to N$.

Note also that given  two maps $f:M\to N$, $g:M' \to N$, an equivariant map of
bibundles $\phi:\graff(f) \to \graff(g)$ has to be of the form $\phi
(x, f(x)) = (h(x), g(h(x)))$ for some map $h: M\to M'$. That is,
$\phi: \graff(f) \rightarrow  \graff(g)$ corresponds to $h:M\to M'$
with the diagram
$\vcenter{
 \xy
(0,6)*+{M}="1"; (0,-6)*+{M'}="2";
(12,0)*+{N} ="3";
{\ar@{->}_{h} "1";"2"};
{\ar@{->}^{f} "1";"3"};
{\ar@{->}_{g} "2";"3"};
\endxy}$
commuting.  This example is also important for the embedding the category
of manifolds into the 2-category of stacks.
\end{example}

\begin{example}\label{ex}
Let $M$ be a manifold and $\Gamma$ a Lie group.  As we have seen a
number of times the manifold $M$ defines the groupoid $\{M\toto M\}$.
The group $\Gamma$ defines the action groupoid $\{\Gamma \toto *\}$ for
the action of $\Gamma$ on a point $*$.  A bibundle $P:\{M\toto M\} \to
\{\Gamma\toto *\}$ is a principal $\Gamma$-bundle over $M$.  A
bibundle $P$ is isomorphic to a bibundle of the
form $\langle f\rangle$ for some functor $f: \{M\toto M\} \to \{\Gamma
\toto *\}$ only if it has a global section, that is, only if it is
trivial.  Thus {\em there are many more bibundles than functors.}

Note, however, that any principal $\Gamma$-bundle $P\to M$ is {\em
locally } trivial.  Hence, after passing to an appropriate cover
$\phi: \cU\to M$, the bibundle $\phi^*P: \{\cU\times _M \cU\toto \cU\}
\to \{\Gamma \toto *\}$ is isomorphic to $\langle f\rangle $ for some
functor $f: \{\cU\times _M \cU\toto \cU\} \to \{\Gamma \toto *\}$.
This is a special case of Lemma~\ref{lem:fact-bibundle-functors}
below.

Note also that the functor $f :\{\cU\times _M \cU\toto \cU\} \to
\{\Gamma \toto *\}$ is a \v{C}ech 1-cocycle on $M$ with coefficients in $\Gamma$ with
respect to the cover $\cU$.
\end{example}

\begin{remark}
Bibundles can be composed: if $P:G\to H$ and $Q:H\to K$ are
bibundles, we define their composition to be the quotient of the
fiber product $P \times_{H_0} Q$ by the action of $H$:
\[
Q\circ P:= (P \times_{H_0} Q)/H.
\]
This makes sense: Since $Q\to H_0$ is a principal $K$-bundle, the
fiber product $P \times_{H_0} Q$ is a manifold.  Since the action of
$H$ on $P$ is principal, the action of $H$ on $P \times_{H_0} Q$
given by $(p, q) \cdot h = (p\cdot h, h\inv \cdot q)$ is free and
proper.   Hence the quotient $(P \times_{H_0} Q)/H$
is a manifold.  Since the action of $H$ on $P \times_{H_0} Q$
commutes with the actions of $G$ and $K$, the quotient $(P
\times_{H_0} Q)/H$ inherits the actions of $G$ and $K$.  Finally,
since $Q\to H_0$ is a principal $K$-bundle, $(P \times_{H_0} Q)/H
\to G_0$ is a principal $K$-bundle.
\end{remark}

\begin{remark}
The composition of bibundles is not strictly associative: if $P_1,
P_2, P_3$ are three bibundles then $P_1 \circ (P_2 \circ P_3) $ is not
the same manifold as $(P_1 \circ P_2) \circ P_3$.  On the other hand
the two bibundles are {\em isomorphic} in the sense of
Definition~\ref{def:iso-bibund}: there is an
equivariant diffeomorphism $\alpha :P_1 \circ (P_2 \circ P_3) \to (P_1
\circ P_2) \circ P_3$.  This is the reason why we end up with a weak
2-category when we replace functors by bibundles.
\end{remark}

\begin{remark}
A natural transformation $\alpha: f\Rightarrow g$ between two functors
$f,g:K\to L$ gives rise to an isomorphism $\langle \alpha \rangle:
\langle f \rangle
\to \langle g \rangle $ of the corresponding bibundles.
\end{remark}

\begin{remark}
If a bibundle $P:G\to H$ is $G$-principal, then it defines a bibundle
$P\inv:H \to G$: switch the anchor maps, turn the left $G$-action into
the right $G$-action and the right $H$-action into a left
$H$-action. Indeed, the compositions $P\inv \circ P$ and $P\inv \circ
P$ are isomorphic to $\langle id_G\rangle $ and $\langle id_H\rangle $
respectively.
\end{remark}

\noindent
We summarize (without proof):
\begin{enumerate}
\item The collection (Lie groupoids, bibundles, isomorphisms of
bibundles) is a weak 2-category.  We denote it by $\Bi$.
\item The strict 2-category of Lie groupoids, smooth functors and natural
transformations embeds into $\Bi$.  For this reason bibundles are
often refered to as ``generalized morphisms."
\end{enumerate}

The lemma below allows us to start justifying our notions of
equivalence of Lie groupoids.
\begin{lemma}\label{lem:equiv-invert}
A functor $f:G\to H$ is an equivalence of Lie groupoids iff the
corresponding bibundle $\langle f \rangle: G\to H$ is $G$-principal,
hence (weakly) invertible.
\end{lemma}
\begin{proof}
Recall that a functor $f:G\to H$ is an equivalence of Lie groupoids iff two conditions hold (cf. Definition~\ref{def:equiv-Lie-gpoid}):
\begin{enumerate}
\item the  map
\[
\varphi : G_1 \to (G_0\times G_0)\times_{(f,f),(H_0,H_0), (s,t)} H_1,
\quad \varphi(\gamma)= (s(\gamma), t(\gamma), f(\gamma))
\]
 is a diffeomorphism and
\item  the map
$b: G_0 \times_{F, H_0, t} H_1 {\to } H_0$, $b(x,h)=
s(h)$ is a surjective submersion.
\end{enumerate}
Recall also that $
\langle f \rangle = G_0 \times_{f, H_0, t} H_1$ and that the right anchor
$a_R: \langle f\rangle \to H_0$ is precisely the map $b$, while the
left anchor is the projection on the first factor: $a_L (x,h) = x$.
Tautologically $a_R$ is a surjective submersion iff $b$ is a
surjective submersion.

Suppose that $G$ acts freely and transitively on the fibers of $a_R:
\langle f\rangle \to H_0$.  That is, suppose $a_R: \langle f\rangle
\to H_0$ is a principal $G$-bundle. Then the map
\[
\psi: G_1\times _{s, G_0 , a_L} ( G_0 \times_{f, H_0, t} H_1)
\to \langle f \rangle \times _{H_0}\langle f \rangle
\quad \psi (g, x, h) =((x, h), (t(g), f(g) h'))
\]
is a diffeomorphism.  Hence it has a smooth inverse.  Thus for any
$(x,h), (x',h') \in G_0\times H_1$ with $f(x) = t(h)$, $f(x') = t(h')$
and $s(h) = s(h')$ there is a unique $g\in G_1$ depending smoothly on
$x,x',h$ and $h'$ with $s(g) =x$, $t(g) = x'$ and $h' = f(g)h$.
Therefore for any $x,y\in G_0$ and any $h'\in H_1$ with $s(h') = f(x)$
and $t(h') = f(y)$ there is a unique $g\in G_1$ depending smoothly on
$x,y$ and $h'$ so that $h' = f(g)1_{f(x)}$.  That is, the map
\[
\varphi: G_1 \to (G_0 \times G_0) \times _{(f,f), H_0\times H_0, (s,t)}H_1
\]
has a smooth inverse.  Therefore if $\langle f\rangle \to H_0$ is left
$G$-principal bundle then $f$ is an equivalence of Lie groupoids.

Conversely suppose $\varphi$ has a smooth inverse.  Then for any
$((x,h), (x',h')) \in \langle f\rangle \times _{H_0} \langle f\rangle$
there is a unique $g\in G_1$ with $s(g) = x'$, $t(g) = x$ and $f(g) =
h (h')\inv$.  Hence the map $\psi $ has a smooth inverse.  Therefore,
if $f:G\to H$ is an equivalence of Lie groupoids, then $\langle
f\rangle \to H_0$ is left $G$-principal bundle.
\end{proof}

\begin{corollary}
Let $G$ be a groupoid and $\phi: \cU\to G_0$ a cover (a surjective
local diffeomorphism).  The the bibundle $\langle \tilde{\phi}\rangle$
defined by the induced functor $\tilde{\phi}: \phi^*G \to G$ is
invertible.
\end{corollary}
\begin{proof}
We have seen  that the functor $\tilde{\phi}: \phi^*G \to G$
is an equivalence.  The result follows from
Lemma~\ref{lem:equiv-invert} above.
\end{proof}

\begin{lemma}\label{lem:global-sec}
Let $P:G\to H$ be a bibundle from a groupoid $G$ to a groupoid $H$.
Then $P$ is isomorphic to $\langle f\rangle$ for some functor $f:G\to
H$ if and only if $a_L:P\to G_0$ has a global section.
\end{lemma}
\begin{proof}
We have seen that for a functor $f:G\to H$ the map $a_L: G_0\times
_{H_0}H_1 \to G_0$ has a global section.

Conversely, suppose we have a bibundle $P:G\to H$ and the principal
$H$-bundle$a_L: P\to G_0$ has a global section. Then by
Lemma~\ref{lem:princH-global} the bundle $P\to G_0$ is isomorphic to
$G_0 \times _{\phi, H_0, t} H_1$ for some map $\phi:G_0 \to H_0$.  Therefore 
we may assume that $P=G_0 \times _{\phi, H_0, t} H_1$.  Now the left
action of $G$ on $P$ defines a map $f:G_1 \to H_1$ by
\[
g\cdot (t(g) 1_{\phi(t(g)} ))= (s(g), 1_{\phi(s(g))})\cdot f(g).
\]
The map $f$ is well defined since the action of $H$ is principal.
Finally the map $f$ preserves multiplication:
if $z\stackrel{g_2}{\to } y \stackrel{g_1 }{\to } x$ are two composable arrows in $G_1$ then, on one hand,
\[
g_2 \ (g_1 \cdot (x, 1_{\phi(x)} ) = g_2 \cdot ( y, 1_{\phi(y)} ) \cdot
f(g_1) = ((z, 1_{\phi (z)}) \cdot f(g_2) )\cdot f(g_1)
\]
and on the other
\[
(g_2 g_1)  \cdot (x, 1_{\phi(x)} ) = (z, 1_{\phi(z)})\cdot f(g_2g_1).
\]
Hence $f(g_2) f(g_1) = f(g_2 g_1)$, that is, $f$ is a functor.
\end{proof}

\begin{lemma} \label{lem:fact-bibundle-functors}
Let $P:G\to H$ be a bibundle from a groupoid $G$ to a groupoid $H$.
There is a cover $\phi: \cU\to G_0$ and a functor $f: \phi^*G \to H$
so that
\[
    P\circ \langle \tilde{\phi}\rangle \stackrel{\simeq}{\to} \langle f\rangle,
\]
where $\tilde{\phi}: \phi^*G \to G$ is the induced functor and
$\stackrel{\simeq}{\to}$ an isomorphism of bibundles.
\end{lemma}

\begin{proof}
Since $a_L: P\to G_0$ is an $H$-principal bundle, it has local
sections $\sigma_i: U_i \to P$ with $\bigcup U_i = G$.  Let $\cU =
\bigsqcup U_i$ and $\phi: \cU\to G_0$ be the inclusion.  Then
$\phi^*P\to \cU$ has a global section.  Hence, by
Lemma~\ref{lem:global-sec} there is a functor $f:\phi^*G \to H$ with
$\langle f\rangle =\phi^*P$.
\end{proof}

\subsection{Hilsum-Skandalis category of Lie groupoids}

Recall that $\Bi$ denotes the weak 2-category with objects Lie
groupoids, 1-arrows bibundles and 2-arrows equivariant maps between
bibundles.  The 2-arrows are always invertible.  Recall that $\Gr$
denotes the (2-)category of Lie groupoids, functors and natural
transformations.
\begin{definition}
Define the 1-category $\Gp$ to be the category with objects Lie
groupoids and arrows the {\em isomorphism classes} $[f]$ of smooth
functors.

Define the 1-category $\HS$ (for Hilsum and Skandalis
\cite{HilsumSkan}, who invented it) to be the category constructed out
of $\Bi$ by identifying isomorphic bibundles.
\end{definition}
 There is an evident
functor $\tilde{z}: \Gr \to \HS$ which is the identity on objects and
takes a functor $f$ to the equivalence class of the bibundle $\langle
f\rangle $ defined by $f$: $\tilde{z}(f) = [\langle f\rangle ]$.
Clearly it drops down to a faithful functor
\[
z: \Gp \to \HS, \quad z(G\stackrel{[f]}{\to} H) =
(G\stackrel{[\langle f\rangle ]}{\to} H).
\]
By abuse of notation let $W$ denote the collection of isomorphism
classes of equivalences in $\Gp$:
\[
W = \{ [w]\mid w\in \Gr_1 \textrm{ is an equivalence}\}
\]

\begin{proposition}\label{prop:HSloc}
The functor $z: \Gp \to \HS$ defined above localizes $\Gp$ at the class
of equivalences $W$.  That is, $z$ induces an equivalence of
categories $\Gp[W\inv] \to \HS$.
\end{proposition}

\begin{proof}
By Lemma~\ref{lem:equiv-invert}, $z([w])$ is invertible in $\HS$ for
any equivalence $w$.  Thus the content of the Proposition is the
universal property of the functor $z: \Gp \to HS$.  Suppose $\Phi: \Gp
\to \sfE$ is a functor that sends isomorphism classes of equivalences
to invertible arrows.  We want to construct a functor $\Psi: HS \to \sfE$ so that
\[
  \Psi \circ z = \Phi.
\]
As the first step, for an object $G\in \HS_0$ define $\Psi (G) = \Phi (G)$.
Next let $P: G\to H$ be a bibundle.  We want to define $\Psi
([P])$. By Lemma~\ref{lem:fact-bibundle-functors} we can factor $P$ as
\[
P \simeq \langle f' \rangle\circ \langle w' \rangle \inv
\]
for some equivalence $w':G' \to G$ and a functor $f': G'\to G$.
Define
\[
\Psi ([P]) = \Phi([ f' ]) \Phi([w' ]) \inv .
\]
We need to check that this is well defined and that $\Psi$ preserves
compositions.
Suppose $w'':G'' \to G$ and $f'':G''\to G$ is another choice of factorization.
Let
\[
[Q] = z[w']\inv z[w'] :G'' \to G'.
\]
Then $[Q]$ can be factored as well:
\[
[Q] = z([g]) z([\tilde{w}])\inv
\]
for some equivalence $\tilde{w}:\tilde{G} \to G''$ and some functor
$g:\tilde{G} \to G'$.
The diagram
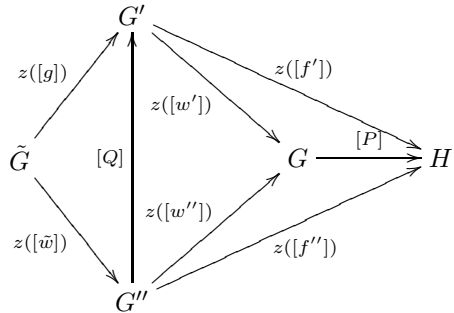
\[
\xy
(-16,0)*+{\tilde{G}}="1";
(-1,19)*+{G'}="2";
(-1,-19)*+{G''}="3";
(21,0)*+{G}="4";
(40,0)*+{H}="5";
{\ar@{->}^{z([g])}  "1";"2"}; {\ar@{->}_{z([w'])}  "2";"4"};
{\ar@{->}_{z([\tilde{w}])}  "1";"3"}; {\ar@{->}^{[Q]}  "3";"2"};
{\ar@{->}^{z([w''])}  "3";"4"};{\ar@{->}^{[P]}  "4";"5"};
 {\ar@{->}^{z([f'])}  "2";"5"};
{\ar@{->}_{z([f''])} "3";"5"};
\endxy
\]
commutes in $\HS$.  Hence
\begin{equation}\label{eq4.1}
z([f'']) z([\tilde{w}]) = z([f']) z ([g]).
\end{equation}
Since $z$ is faithful,
\[
[f''][\tilde{w}] = [f'][g]
\]
in $\Gp$.  Hence, in $\sfE$,
\[
\Phi([f'']) \Phi([\tilde{w}] = \Phi([f'])\Phi([g])= \Phi([f'])\Phi([w'])\inv \Phi([w'])\Phi([g] = \Phi([f']) \Phi([w'])\inv \Phi ([w'']) \Phi([\tilde{w}]),
\]
where we used the fact that $z$ is faithful and (\ref{eq4.1}).
Since $\Phi([\tilde{w}])$ is invertible,
\[
\Phi([f'']) = \Phi([f']) \Phi([w'])\inv \Phi ([w'']).
\]
Therefore
\[
\Phi([f'']) \Phi ([w''])\inv = \Phi([f']) \Phi([w'])\inv,
\]
and $\Psi$ is well-defined.

A similar argument shows that $\Psi$ preserves multiplication.
\end{proof}

\begin{definition}[Morita equivalent groupoids]  Two Lie groupoids are
{\em Morita equivalent} if there they are isomorphic in the
localization $\Gp[W\inv]$ of the category of groupoids at
equivalences.  In particular, $G$ and $H$ are Morita equivalent if
there is a bibundle $P:G\to H$ with the action of $G$ being principal.
\end{definition}

We finally come to the punchline of the section: the localization of
the category of Lie groupoids at equivalences as a
\emph{1-category} has problems.
\begin{lemma}\label{lem:trouble}
There is a cover $\{U_1, U_2\}$ of $S^1$ and two morphisms $f,g: S^1
\to \{\Z/2\toto *\}$ in $\Gp[W\inv]$ so that $f|_{U_i} = g|_{U_i}$
($i=1, 2$) but $f\not = g$.
\end{lemma}

\begin{proof}
In the category $\HS$ a morphism from a manifold $M$ (that we think of
as the groupoid $\{M\toto M\}$) to a groupoid $G$ is the equivalence
class of a bibundle $P$ from $\{M\toto M\}$ to $G$.  An action of
$\{M\toto M\}$ on $P$ is simply a map $a_L:P\to M$.  So a bibundle
from $M$ to $G$ is a principal $G$ bundle and an $\HS$ morphism from
$M$ to $G$ is the equivalence class of some principal $G$-bundle over
$M$.  Hence an $\HS$ morphism from $S^1$ to $\{\Z/2\toto *\}$ is the
class of a principal $\Z/2$ bundle over $S^1$ (cf. Example~\ref{ex}).
There are two such classes: the class of the trivial bundle $a$ and
the class of the nontrivial bundle $b$.  Now cover $S^1$ by two
contractible open sets $U_1$ and $U_2$.  Any principal $S^1$ bundle
over a contractible open set is trivial.  Therefore $a|_{U_i} =
b|_{U_i}$, $i=1,2$.  This gives us the two morphisms in $\HS$ from
$S^1$ to $\{\Z/2\toto *\}$ with the desired properties.  Let $F: \HS
\to \Gp[W\inv]$ denote an equivalence of categories, which exists by
Proposition~\ref{prop:HSloc}.  Then $f= F(a)$ and $g= F(b)$ are the
desired morphisms in $\Gp[W\inv]$.
\end{proof}

It may be instructive to note how this problem does not arise in the
weak 2-category $\Bi$.  In $\Bi$ the 1-arrows are not isomorphism
classes of bibundles but actual bibundles.  Let $P_1 \to S^1$ denote a
trivial $\Z/2$ principal bundle and $P_2 \to S^1$ a nontrivial one.
Over the open sets $U_1$, $U_2$ we have isomorphisms $\varphi_i:
P_1|_{U_i} \stackrel{\simeq}{\to} P_2|_{U_i}$, rather than equalities,
as we had with their isomorphisms classes.  These local isomorphisms
obviously do not glue together to form a global isomorphism from $P_1$
to $P_2$.  They can't, because $P_1$ and $P_2$ are not isomorphic.
And they don't because they don't agree on double intersections:
$\varphi_1|_{P_1|_{U_1 \cap U_2} } \not = \varphi_2|_{P_1|_{U_1 \cap
U_2} }$.\\

At this point we can agree that the right setting for orbifolds is the
weak 2-category $\Bi$ and declare our mission accomplished.  That is,
a smooth orbifold would be a Lie groupoid weakly isomorphic in $\Bi$
(i.e., Morita equivalent) to a proper etale Lie groupoid.  We would
call such groupoids {\em orbifold groupoids}. A map between two
orbifolds would be a smooth bibundle.  

The geometry of orbifolds would proceed along the lines of Moerdijk's
paper \cite{moer}.  For example, let us define ``vector
orbi-bundles.''  The definition is modeled on the case where the
orbifold is a manifold with an action of a finite group.  That is,
suppose a finite group $\Gamma$ acts on a manifold $M$.  A vector
bundle over the orbifold ``$M/\Gamma$'' is a $\Gamma$-equivariant
vector bundle $E\to M$.  Hence, in general, a vector bundle over an
orbifold groupoid $G$ is a vector bundle $E\to G_0$ over the space of
objects together with a linear left action of $G$ on $E$ (``linear''
means that the map $G_1 \times _{G_0} E \to E$ is a vector bundle
map).  A bit of work shows that one can pull back a vector bundle by a
bibundle.

On the other hand, there is still something awkward in this set-up,
since the composition of bibundles is not strictly associative.  This
gets particularly strange when we start thinking about flows of vector
fields, or, more generally, group actions.  For example, let the
circle $S^1$ acts on itself by translations.  Now take an open cover
$\cU\to S^1$ and form the cover groupoid $G= \{ \cU\times _{S^1} \cU
\toto \cU\}$.  The induced functor $G \to \{S^1 \toto S^1\}$ is weakly
invertible, so we get an ``action'' of $S^1$ on $G$.  The word
``action'' is in quotation marks because for any two elements of the
group $\lambda, \lambda' \in S^1$ and the corresponding isomorphisms
$\phi_\lambda, \phi_{\lambda'}: G \to G$
\[
\phi_\lambda \circ  \phi_{\lambda'} \not = \phi_{\lambda + \lambda'}.
\]
Rather,
\[
\phi_\lambda \circ  \phi_{\lambda'}
\stackrel{A}{\to} \phi_{\lambda + \lambda'}
\]
for some isomorphism of bibundles $A$ depending on $\lambda,
\lambda'$.  We get a so called {\em weak action} of $S^1$ on $G$.

The same thing happen when we try to integrate a
vector field on $G$: we don't get a flow in the sense of an action of
the reals.  We get some sort of a weak flow.  For the same reason the
action of the Lie algebra $Lie(\Gamma)$ on a proper etale Lie groupoid
$G$ with the compact coarse moduli space $G_0/G_1$ will not integrate
to the action of the Lie group $\Gamma$.  It will only integrate to a
weak action.  This is somewhat embarrassing since in literature Lie
groups routinely act on orbifolds.

There is another question that may be nagging the reader: aren't
groupoids supposed to be atlases on orbifolds, rather than being
orbifolds themselves?  There is a solution to both problems.  It
involves embedding the weak 2-category $\Bi$ into an even bigger
gadget, the 2-category of stacks $\St$.  Stacks form a strict
2-category.  This is the subject of the next and last section.  In
particular in $\St$ the composition of 1-arrows is associative and
strict group actions make perfectly good sense.  Additionally there is
a way of thinking of a groupoid as ``coordinates'' on a corresponding
stack.  Different choices of coordinates define Morita equivalent
groupoids.  And Morita equivalent groupoids define ``the same''
(isomorphic) stacks.

\section{Stacks}\label{sec:stacks}

In section~\ref{bi} we constructed a weak 2-category $\Bi$ whose
objects are Lie groupoids, 1-arrows (morphisms) are bibundles and
2-arrows (morphisms between morphisms) are equivariant maps between
bibundles.  The goal of this section is to describe a particularly
nice and concrete (?!) strictification of this weak 2-category.  That
is, we describe a strict 2-category $\St$ of stacks and a functor $B:
\Bi \to \St$ which is an embedding of weak 2-categories (there is no
established name in literature for this functor, so I made one up).
The 2-category $\St$ of stacks is a sub-2 category of the category of
categories $\Cat$.  Recall that the objects of $\Cat$ are categories,
the 1-arrows are functors and the 2-arrows are natural
transformations.

Here is a description of the 2-functor $B: \Bi \to \Cat$ (it will land
in $\St$ once we define/explain  what $\St$ is):\\

\noindent 1.  To a groupoid $G$
assign  the category $BG$,
whose objects are principal $G$-bundles and morphisms are
$G$-equivariant maps.\\

\noindent 2.  To a bibundle $P:G\to H$  assign a functor
\[
BP :BG \to BH
\]
as follows:  A principal $G$-bundle $Q$ on a manifold $M$ is a
bibundle from the groupoid $\{M\toto M\}$ to $G$.  Define
\[
BP (Q) = P \circ Q \quad \textrm{(a composition of bibundles)}.
\]
A $G$-equivariant map $\phi:Q_1 \to Q_2$ between two principal
$G$-bundles $Q_1 \to M_1$, $Q_2 \to M_2$ induces an $H$-equivariant
map $BP (\phi): P\circ Q_1 \to P\circ Q_2$ between the
corresponding principal $H$-bundles.    It is not hard to check that
$BP$ is actually a functor.\\

\noindent 3.  To a $G$-$H$ equivariant map $A: P\to P'$ assign a natural
transformation $BA :BP \Rightarrow BP'$ as follows.  Given a principal
$G$-bundle $Q$, the map $A: P\to P'$ induces a $G$-$H$ equivariant map
$\tilde{A}: Q\times _{G_0} P \to Q\times _{G_0} P'$ which descends an
$H$ equivariant diffeomorphism
\[
BA (Q): BP (Q) \equiv P\circ Q \equiv (Q\times _{G_0} P)/G \to (Q\times _{G_0}
P')/G \equiv BP' (Q).
\]

\begin{remark}\label{rm:slice-cat}
The notation $B\{M\toto M\}$ is quite cumbersome.  Instead we will use
the notation $\under{M}$.

It follows from Example~\ref{ex:man-as-groupoid} that the category
$\under{M}$ has the following simple description.  It objects are maps
$Y\stackrel{f}{\to} M$ of manifolds into $M$.  A morphism in
$\under{M}$ from $f:Y\to M$ to $f':Y' \to M$ is a map of manifolds
$h:Y\to Y'$ making the diagram $\vcenter{\xy (0,6)*+{Y}="1";
(0,-6)*+{Y'}="2"; (12,0)*+{M} ="3"; {\ar@{->}_{h} "1";"2"};
{\ar@{->}_{f'} "2";"3"};{\ar@{->}^{f} "1";"3"};
\endxy }$ commute.  The category $\under{M}$ is an example of a {\em slice}
(or {\em comma}) category.
\end{remark}

We now proceed to describe the image of the functor $B: \Bi \to \Cat$.
More precisely we will describe a slightly larger 2-category of {\em
geometric stacks} and the functor $B$ will turn out to be an
equivalence of weak 2-categories $B: \Bi \to $ geometric stacks.  More
precisely, we'll see that every geometric stack is isomorphic to a
stack of the form $BG$ for some Lie groupoid $G$.

We define geometric stacks in several step.  We first define
categories fibered in groupoids (CFGs) over the category of manifolds
$\Man$.  Next we define stacks.  These are CFG's with sheaf-like
properties.  Then we single out geometric  stacks.  These
are the stacks that have atlases. Finally  any
geometric stack is isomorphic (as a stack) to a stack of the form $BG$
for some groupoid $G$.

\subsection{Categories fibered in groupoids}
\begin{definition}\label{def:CFG}
A {\em category fibered in groupoids (CFG)}  over a
category $\sfC$ is a functor $\pi :\sfD\to \sfC$ such that
\begin{enumerate}
\item Given an arrow $f:C'\to C$ in $\sfC$ and an object
$\xi\in \sfD$ with $\pi(\xi)
= C$ there is an arrow $\tilde{f}:\xi' \to \xi$ in $\sfD$ with $\pi
(\tilde{f}) = f$ (we think of $\xi'$ as a {\em pullback} of $\xi$ along $f$).
\item  Given a diagram $\vcenter{ \xymatrix@=8pt@ur{
& \xi''\ar[d]^f\\ \xi' \ar[r]_{h} & \xi}}$ in $\sfD$ and a diagram
$\vcenter{ \xymatrix@=8pt@ur{ & \pi(\xi'')\ar[dl]_{g}\ar[d]^{\pi(f)}\\
\pi(\xi') \ar[r]_{\pi(h)} & \pi(\xi)}}$ in $\sfC$ there is a unique
arrow $\tilde{g}:\xi''\to \xi'$ in $\sfD$ making $\vcenter{
\xymatrix@=8pt@ur{ & \xi''\ar[d]^f\ar[dl]_{\tilde{g}}\\ \xi'
\ar[r]_{h} & \xi}}$ commute and satisfying $\pi (\tilde{g}) = g$.
That is, there is a unique way to fill in the first diagram so that
its image under $\pi$ is the second diagram.
\end{enumerate}
We will informally say that $\sfD$ is a category fibered in groupoids
over $\sfC$, with the functor $\pi$ understood.
\end{definition}

\begin{example}
Fix a Lie groupoid $G$. I claim that the functor $\pi:BG \to \Man$
that sends a principal $G$ bundle to its base and a $G$-equivariant
map between two principal $G$-bundles to the induced map between their
bases makes the category $BG$ into a category
fibered in groupoids over the category $\Man$ of manifolds.

Indeed condition (1) of Definition~\ref{def:CFG} is easy to check.
Given a map $f:N\to M$ between two smooth manifolds and a principal
$G$-bundle $\xi \to M$ we have the pullback bundle $f^* \xi \to N$
and a $G$-equivariant map $\tilde{f}:f^*\xi \to \xi$ inducing $f$ on
the bases of the bundles.

Note that if $\pi':\xi' \to N$ is a principal $G$-bundle and $h:
\xi' \to \xi$ is a $G$-equivariant map inducing $f:N\to M$ then
there is a canonical $G$-equivariant map $\eta: \xi'\to f^*\xi$ which
is given by $\eta (x) = (\pi' (x), h(x))$.  By
Corollary~\ref{cor:g-equiv-iso}, the map $\eta$ is a diffeomorphism.

To check condition (2) suppose that we have three principal $G$-bundles
$\xi'' \to M''$, $\xi' \to M'$, $\xi \to M$, two $G$-equivariant maps
$f:\xi'' \to \xi$, $h:\xi' \to \xi$ inducing $\bar{f}:M'' \to M$ and
$\bar{h}: M'\to M$ respectively and a map $g:M'' \to M$ so that
$\vcenter{ \xymatrix@=8pt@ur{ & M''\ar[dl]_{g}\ar[d]^{\bar{f}}\\ M'
\ar[r]_{\bar{h}} & M}}$ commutes. We want to construct a
$G$-equivariant map $\tilde{g}:\xi'' \to \xi'$ with $h\circ \tilde{g}
=f$.  By the preceding paragraph we may assume that $\xi'' =
\bar{f}^* \xi = M''\times _M \xi$ and $\xi' = \bar{h}^* \xi =
M'\times _M \xi$.  Define $\tilde{g}:M''\times _M \xi \to M'\times
_M \xi$ by $\tilde{g} (m, x) = (g(m), x)$.  Hence $h\circ \tilde{g}
=f$, and we have verified that $\pi:BG\to \Man$ is a CFG.
\end{example}

\begin{definition}[Fiber of CFG] Let $\pi:\sfD\to \sfC$ be a category fibered
in groupoids and $C\in \sfC_0$ an object.  The {\em fiber }
of $\sfD$ over $C$ is the category $\sfD (C)$ with
objects
\[
\sfD (C)_0 := \{\xi \in \sfD_0 \mid \pi (\xi) = C\}\]
 and arrows/morphisms
\[
\sfD (C)_1 :=\{ (f:\xi' \to \xi) \in \sfD _1 \mid
 \xi, \xi' \in \sfD (C) _0 \textrm{ and } \pi (f) = id_C\}.
\]
\end{definition}

\begin{example}
In the case of $\pi:BG\to \Man$ the fiber of $BG$ over a manifold $M$
is the category of principal $G$-bundles over $M$ and gauge
transformations ($G$-equivariant diffeomorphisms covering the identity
map on the base).
\end{example}
\begin{remark}
Let $\pi:\sfD\to \sfC$ be a CFG.  Suppose $Y\stackrel{f}{\to}X$ is an
arrow in $\sfC$, $\xi\in \sfD(X)_0$, $\xi_1, \xi_2\in \sfD(Y)_0$ and
$h_i :\xi_i \to \xi$ ($i=1,2$) are two arrows in $\sfD$ with $\pi( h_i
) = f$.  Then by
\ref{def:CFG}~(1) there exist unique arrows $k:\xi_1\to \xi_2$ and $\ell: \xi_2 \to \xi_1$ making the diagrams
\[
\xy
(0,6)*+{\xi_1}="1"; (0,-6)*+{\xi_2}="2";
(12,0)*+{\xi} ="3";
{\ar@{->}_{k} "1";"2"};
{\ar@{->}_{h_2} "2";"3"};{\ar@{->}^{h_1} "1";"3"};
\endxy
\quad \textrm { and } \quad
\xy
(0,6)*+{\xi_1}="1"; (0,-6)*+{\xi_2}="2";
(12,0)*+{\xi} ="3";
{\ar@{->}^{\ell} "2";"1"};
{\ar@{->}_{h_2} "2";"3"};{\ar@{->}^{h_1} "1";"3"};
\endxy
\]
commute, with $\pi (k) = \pi (\ell) =id_Y$.  Then, since $\pi (k\circ
\ell) = id_Y$ and
\[
\xy
(0,6)*+{\xi_1}="1"; (0,-6)*+{\xi_1}="2";
(12,0)*+{\xi} ="3";
{\ar@{->}_{\ell \circ k} "1";"2"};
{\ar@{->}_{h_1} "2";"3"};{\ar@{->}_{h_1} "1";"3"};
\endxy
\]
commutes, we must have $\ell\circ k = id_{\xi_1}$.  Similarly, $k\circ
\ell = id_{\xi_2}$.  We conclude: {\em any two pullbacks of $\xi$
along $Y\stackrel{f}{\to}X$ are isomorphic.}\\

\noindent
\framebox{\begin{minipage}[c]{\textwidth }
From now on, given a CFG $\pi:\sfD \to \sfC$ and $\xi\in \sfD(X)_0$  for each arrow  $Y\stackrel{f}{\to}X \in \sfC_1$  we {\bf choose} an
arrow $\tilde{f}$ in $\sfD$ with target $\xi$.  We denote the source
of $\tilde{f}$ by $f^* \xi$ and refer to it as the {\em pullback of $\xi$
by $f$}.  We always choose $id^*\xi = \xi$.
\end{minipage}
} \\[12pt]

Similarly we can define {\em pullbacks of arrows}: Suppose $(\xi_1
\stackrel{\gamma}{\to} \xi_2) \in \sfD(X)_1$ is an arrow in $\sfD$ and
$(Y\stackrel{f}{\to}X)$ is an arrow in $\sfC$.  We then have a diagram
in $\sfD$:
\begin{equation}\label{eq*}
\xy
(0,6)*+{f^*\xi_1}="1"; (0,-6)*+{f^*\xi_2}="2";
(12,6)*+{\xi_1} ="3";(12,-6)*+{\xi_2} ="4";
{\ar@{->}^{\tilde{f}_1} "1";"3"};
{\ar@{->}_{\tilde{f}_2} "2";"4"};{\ar@{->}_{\gamma} "3";"4"};
\endxy.
\end{equation}
By \ref{def:CFG}~(1) applied to $\vcenter{\xy
(0,6)*+{f^*\xi_1}="1"; (0,-6)*+{f^*\xi_2}="2";
(12,0)*+{\xi_2} ="3";
{\ar@{->}^{\gamma\circ \tilde{f}_1} "1";"3"};{\ar@{->} "2";"3"};
\endxy}$
we get the unique arrow $f^*\gamma: f^*\xi_1 \to f^*\xi_2$ making
(\ref{eq*}) commute.
\end{remark}

\begin{remark}
Similar arguments show that a fiber $\sfD (C)$ of a
category $\sfD$ fibered in groupoids over $\sfC$ is actually a
groupoid.  That is, all arrows in $\sfD(C)$ are invertible.
\end{remark}

\begin{definition}[Maps of CFGs] \label{def:mapCFG}
Let $\pi_{\sfD}:\sfD \to \sfC$ and $\pi_{\sfE}:\sfE \to \sfC$ be two
categories fibered in groupoids.  A {\em 1-morphism} (or a {\em 1-arrow})
  $F:\sfD \to \sfE$ of
CFGs is a functor that commutes with the projections: $\pi_{\sfE}
\circ F = \pi_\sfD$.

A 1-morphism  $F:\sfD \to \sfE$ of
CFGs is an {\em isomorphism} if it is an equivalence of categories.

Given two 1-morphisms $F, F':\sfD \to \sfE$ of CFGs, a {\em 2-morphism}
$\alpha: F\Rightarrow F'$ is a natural transformation from $F$ to
$F'$.
\end{definition}

Thus the collection of all categories fibered in groupoids over a
given category $\sfC$ is a strict 2-category.  Note also that
natural transformations between 1-arrows of CFGs are automatically
invertible since the fibers of CFGs are groupoids.  We note that for
any two CFGs $\sfD$ and $\sfE$ over $\sfC$, the collection of
1-arrows $\Hom (\sfD, \sfE)$ forms a category.  In fact, it is a
groupoid.

\subsection{Descent}

To make sense of the next definition, consider how a principal
$G$-bundle $P\to M$ ($G$ a Lie groupoid) can be reconstructed from its
restrictions to elements of an open cover $\{U_i\}$ of $M$ and the
gluing data.\footnote{The reader may think of $G$ as a Lie group to
avoid getting bogged down in irrelevant technicalities.} We have
restrictions $P_i = P|_{U_i}$ and isomorphisms $P_i|_{U_{ij}} \to
P_j|_{U_{ij}}$ over double intersections $U_{ij} := U_i \cap U_j$
satisfying the cocycle conditions.  Given a $G$-equivariant map
$\phi:P'\to P$ of two principal $G$-bundles covering the identity map
on the base, we have a collection of $G$-equivariant maps $\phi_i:P'_i
\to P_i$ which agree on double intersections: $\phi_i|_{P_i|_{U_{ij}}}
= \phi_j|_{P_j|_{U_{ij}}}$.

Conversely, given a collection of principal $G$-bundles $\{ P_i\to
U_i\}$ and isomorphisms $\theta_{ij}:P_i|_{U_{ij}} \to P_j|_{U_{ij}}$
satisfying the cocycle conditions, there is a principal $G$-bundle $P$
over $M$ with $P|_{U_i}$ isomorphic to $P_i$ for all $i$.  Similarly,
given two collections $( \{P'_i\to U_i\}, \{\theta'_{ij}:P_i|_{U_{ij}}
\to P_j|_{U_{ij}}\})$, $( \{P_i\to U_i\}, \{\theta_{ij}:P_i|_{U_{ij}}
\to P_j|_{U_{ij}}\})$ and a collection of principal $G$-bundle maps
$\{\phi_i:P'_i \to P_i\}$ compatible with $\{\theta'_{ij}\}$ and
$\{\theta'_{ij}\}$, there is a $G$-equivariant map $\phi: P' \to P$
which restricts to $\phi_i$ over $U_i$.

A succinct way of describing the above local-to-global correspondence
is through the language of equivalences of categories.    We have the
category $BG(M)$ of principal
$G$-bundles over $M$ and $G$-equivariant maps covering $id_M$.  We may
think of it as the category $\Bi(\{M\toto M\}, G)$ of bibundles from
$\{M\toto M\}$ to $G$.  Given a cover $\cU = \bigsqcup U_i \to M$, we
have the cover groupoid $\cU\times _M \cU \toto \cU$.  A collection $(
\{P_i\to U_i\}, \{\theta_{ij}:P_i|_{U_{ij}}
\to P_j|_{U_{ij}}\})$ of principal $G$-bundles is nothing but a bibundle
from the cover groupoid to $G$.  Similarly, a map between two such
collections is an equivariant map between two bibundles.  And the
restriction map $P\mapsto \{P|_{U_i}\}$ induces a map between the
two categories:
\[
\Psi:
\Bi(\{M\toto M\}, G) \to \Bi( \{\cU\times _M \cU \toto \cU\}, G).
\]
Formally, on objects,
\[
\Psi (Q) = Q\circ U,
\]
where $U: \{\cU\times _M \cU \toto \cU\} \to \{M\toto M\}$ is the
bibundle with the total space $\cU$, left anchor the identity map and
the right anchor the ``embedding'' $\cU \to M$.  Since a
$G$-equivariant map $Q\to Q'$ induces a $G$-equivariant map $Q\circ
U\to Q'\circ U$, $\Psi$ is a functor.  Moreover, since $U$ is weakly
invertible, $\Psi $ is an equivalence of categories.
One says that the principal $G$-bundles on the cover $\cU$ satisfying the
compatibility conditions {\em descend} to the principal $G$-bundles on $M$.

More generally, given a CFG $\pi: \sfD\to \Man$ and a cover $\cU\to
M$, one defines the {\em descent category } $\sfD(\cU\to M)$.  To do
it properly, we need to correct one inaccuracy in the discussion
above.  We have taken advantage of the fact that one can restrict
principal bundles to open sets.  Furthermore if $\{U_i\}$ is a cover
of a manifold $M$ and $P\to M$ a principal $G$-bundle, then
$(P|_{U_i})|_{U_{ij}} = P|_{U_{ij}} = (P|_{U_j})|_{U_{ij}}$ (here,
again, $U_{ij} = U_i\cap U_j$).  But if we want to think of $BG\to
\Man $ abstractly, as a CFG, then restrictions should be replaced by
pullbacks.

Now if if $M''\stackrel{g}{\to} M' \stackrel{f}{\to }M$ are maps of
manifold and $\xi$ is an object of $\sfD$ over $M$, then we don't
expect $(f\circ g)^*\xi$ to equal $g^* (f^* \xi)$; we only expect them
to be canonically isomorphic.  And indeed if $\sfD = BG$ so that $\xi$
is a principal $G$-bundle, then the pullback $f^*(g^*P)$ is {\em not}
the same as $(f\circ g)^* P$ even as a {\em set!}  To talk about
descent in general we need to replace restrictions by pull-backs:
instead of $P|_{U_i}$ we should think $\iota_i^*P$ where $\iota _i :U_i
\to M$ denotes the canonical inclusion.   We will then discover that
$\iota_{ij}^* \iota_i^*P $ is isomorphic but not equal to
$\iota_{ji}^* \iota_j^* P$ ($\iota_{ij}$ and $\iota_{ji}$ denote the
inclusions of the double intersection $U_{ij}$ into $U_i$ and $U_j$
respectively), so the bookkeeping gets a bit more complicated.  Let us
now properly organize all this bookkeeping. We closely follow Vistoli
\cite{vistoli}.

Given an open covering $\{U_i \hookrightarrow M\}$ of a manifold $M$ we
think of the double intersections $U_{ij} = U_i\cap U_j$ as fiber
products $U_i\times _M U_j$ and triple intersections $U_{ijk} $ as
fiber products $U_i \times _M \times U_j \times _M U_k$.  Let $\pr_1
\colon U_i \times_M U_j \to U_i$ and $\pr_2 \colon U_i \times_M U_j
\to U_j$ the first and second projection respectively.  Similarly for any 
three indices $i_1, i_2, i_3$ we have projection $p_a: U_{i_1} \times
_M \times U_{i_2} \times _M U_{i_3} \to U_{i_a}$, $a=1,2,3$.  We also
have a commuting cube:

\begin{equation}\label{eq:basic-cube}
   \xymatrix{
   &U_{ijk}\ar[rr]^{\pr_{23}}\ar'[d]^{\pr_{13}}[dd]
   \ar[ld]_{\pr_{12}}
   &&U_{jk}\ar[ld]\ar[dd]\\
   U_{ij}\ar[rr]\ar[dd]
   && U_j \ar[dd]\\
   &U_{ik} \ar[ld]\ar'[r][rr]
   &&U_k\ar[ld]\\
   U_i \ar[rr]&& M
   }
\end{equation}
where $\pr_{12}$, $\pr_{13}$ and $\pr_{23}$ denote the appropriate
projections.

\begin{definition}[Descent category]
Let $\pi:\sfD\to \Man$ be a category fibered in groupoids, $M$ a
manifold and $\{U_i\}$ an open cover of $M$.  An \emph{object with
descent data}$(\{\xi_i\}, \{\phi_{ij}\})$ on $M$, is a collection of
objects $\xi_i \in \sfD(U_i)$, together with isomorphisms $\phi_{ij}
\colon \pr_2^* \xi_j \simeq \pr_1^* \xi_i$ in $\sfD(U_{ij}) =
\sfD(U_i \times_M U_j)$, such that the following cocycle condition is
satisfied: for any
triple of indices $i$, $j$ and $k$, we have the equality \[
\mathrm{pr}_{13}^* \phi_{ik} = \mathrm{pr}_{12}^* \phi_{ij} \circ
\mathrm{pr}_{23}^* \phi_{jk} \colon \pr_3^*\xi_k \to\pr_1^*\xi_i \]
where the $\pr_{ab}$ and $\pr_a$ are projections discussed above.  The
isomorphisms $\phi_{ij}$ are called \emph{transition isomorphisms} of
the object with descent data.

An arrow between objects with descent data
   \[
   \{\alpha_i\} \colon (\{\xi_i\}, \{\phi_{ij}\}) \to
   (\{\eta_i\}, \{\psi_{ij}\})
   \]
is a collection of arrows $\alpha_i\colon \xi_i \to \eta_i$ in $\sfD(U_i)$, 
with the property that for each pair of indices $i$, $j$, the diagram
   \[
   \xymatrix@C+15pt{
   {}\pr_2^* \xi_j \ar[r]^{\pr_2^* \alpha_j} \ar[d]^{\phi_{ij}}
   & {}\pr_2^* \eta_j\ar[d]^{\psi_{ij}}\\
   {}\pr_1^* \xi_i \ar[r]^{\pr_1^* \alpha_i}&
   {}\pr_1^* \eta_i
   }
   \]
commutes.

There is an obvious way of composing morphisms, which makes objects
with descent data the objects of a category, {\em the descent category
of } $\{U_i \to M\}$.  We denote it by  $\sfD(\{U_i \to M\})$,
\end{definition}

\begin{remark}
As before let $\pi:\sfD\to \Man$ be a category fibered in groupoids, $M$ a manifold and $\{U_i\}$ an open cover of $M$.  We have a functor
\[
\sfD(M) \to \sfD(\{U_i \to M\})
\]
 given by pullbacks.
\end{remark}

We are now in position to define stacks over manifolds.
\begin{definition} (Stack)
A category fibered in groupoids $\pi: \sfD\to \Man$ is a {\em stack}
if for any manifold $M$ and any open cover $\{U_i \to M\}$ the pullback functor
\[
\sfD(M) \to \sfD(\{U_i \to M\})
\]
 is an equivalence of categories.
\end{definition}

\begin{example}
The CFG $BG\to \Man$ is a stack for any Lie groupoid $G$.
\end{example}

\begin{example}
Let $\Gamma$ be a Lie group.  The category $dB\Gamma$ with objects
principal $\Gamma$-bundles {\em with connections} and morphisms
connection preserving equivariant maps is a stack.
\end{example}

\begin{definition}(Maps of stacks)
Let $ \pi_\sfC: \sfC \to \Man$, $\pi_\sfD: \sfD: \sfD \to \Man$ be two
stacks. A functor $f: \sfC\to \sfD$ is a {\em map of stacks} (more
precisely a 1-arrow in the 2-category $\St$ of stacks) if it is a map of
CFGs (cf. Definition~\ref{def:mapCFG})---  $f$ commutes with the projections to $\Man$:
\[
\pi_\sfD \circ f = \pi_\sfC.
\]
\end{definition}

\begin{lemma}\label{lem:map-manifold-BH}
Let $M$ be a manifold, $H$ a groupoid. Then any map of stacks
$F:\under{M} \to BH$ is naturally isomorphic to the functor $BP$ induced by a
principal $H$-bundle $P$ over $M$.
\end{lemma}

\begin{proof}
As we have seen in Remark~\ref{rm:slice-cat}, the objects of the CFG
$\under{M}$ are maps $Y\stackrel{f}{\to} M$.  An arrow in $\under{M}$  from
$Y\stackrel{f}{\to} M$ to $Y'\stackrel{f'}{\to} M$ is a
 commuting triangle $\vcenter{ \xymatrix@=8pt@ur{ &
Y\ar[dl]_{h}\ar[d]^f\\ Y' \ar[r]_{f'} & M}} $.  The functor $F$
assigns to each object $Y\stackrel{f}{\to} M$ of $\under{M}$ a
principal $H$-bundle $F(Y\stackrel{f}{\to} M)$ over $M$. Let $P =
F(M\stackrel{id}{\to} M)$.  Note that any map $f:Y\to M$ is {\em also
an arrow in }$\under{M}$: it maps $Y\stackrel{f}{\to}M$ to
$M\stackrel{id}{\to} M$, since $\vcenter{ \xymatrix@=8pt@ur{ &
Y\ar[dl]_{f}\ar[d]^f\\ M \ar[r]_{id} & M}} $ commutes.  Hence we get a
map of principal $H$-bundles
\[
F\left( \vcenter{ \xymatrix@=8pt@ur{  &  Y\ar[dl]_{f}\ar[d]^f\\
M  \ar[r]_{id} & M}}\right) : F(Y\stackrel{f}{\to} M) \to P
\]
projecting down to the map $f:Y\to M$ in $\Man$.  But $BH\to \Man$ is
a CFG and $f^*P\to P$ is another arrow in $BH$ projecting down to
$f:Y\to M$.  Consequently the principal $H$-bundle
$F(Y\stackrel{f}{\to} M) \to Y$ is isomorphic to the bundle $f^*P \to
Y$.  Denote this isomorphism by $\alpha (f)$.  Varying $f\in
(\under{M})_0$ we get a map
\[
\alpha : (\under{M})_0 \to (BH)_1;
\]
it is a natural isomorphism of functors $\alpha : F \Rightarrow BP$.
\end{proof}

\begin{corollary}\label{cor:maps-of-manifolds-as-stacks}
Let $M, M'$ be two manifolds.  For any map $F:\under{M}\to \under{M'}$
of CFGs there is a unique map of manifolds $f:M\to M'$ defining $f$.
That is the functor $\Man \to $ CFG's over $\Man$, $M \mapsto
\under{M}$ is an embedding of categories.
\end{corollary}

\begin{proof}
Any two maps of CFGs from $\under{M}$ to $\under{M'}$ are equal since
the only arrows in the fibers of $\under{M'}$ are the identity arrows.
\end{proof}
\begin{remark}
Note a loss: if we think of smooth manifolds as stacks, we lose the
way to talk about maps between manifolds that are {\em not} smooth.
\end{remark}

\begin{remark}
With a bit of work Lemma~\ref{lem:map-manifold-BH} above can be
improved as follows:\\

Let $G$ and $H$ be two Lie groupoids.  Then any map of stacks $F:BG
\to BH$ is isomorphic to $BP$ for some principal bibundle $P:G\to
H$.\\

Indeed, let $P = F(G_1 \to G_0)$.  It is an object of $BH (G_0)$, that
is, a principal $H$-bundle over $G_0$.  Since $G_1 \to G_0$ also has a
left $G$-action and $F$ is a functor, $P$ also has a left $G$-action.
A bit more work shows that $BP $ is isomorphic to $F$.
\end{remark}

\subsection{2-Yoneda}
Lemma~\ref{lem:map-manifold-BH} generalizes to arbitrary categories
fibered in groupoids.  The result is often refered to as 2-Yoneda
lemma.

For any category $\sfC$ and any object $C\in \sfC_0$ there
exists a CFG $\underline{C}$ over $\sfC$ defined as follows. The
objects of $\underline{C}$ are maps $C' \stackrel{f}{\to } C\in
\sfC_1$.  A morphism from $C' \stackrel{f}{\to } C$ to $C''
\stackrel{g}{\to } C$ is a commuting triangle $\vcenter{
 \xymatrix@=8pt@ur{ & C'\ar[dl]_{h}\ar[d]^f\\ C'' \ar[r]_{g} & C}} $.
 There is an evident composition of such triangles (stick them
 together along the common side) making $\underline{C}$ into a
 category.  There is also a functor $\pi_C : \under{C}\to \sfC$:
 $\pi_C (C' \stackrel{f}{\to } C) = C'$ and $\pi_C (\vcenter{
 \xymatrix@=8pt@ur{ & C'\ar[dl]_{h}\ar[d]^f\\ C'' \ar[r]_{g} & C}} ) =
 (h:C' \to C'')$.

\begin{lemma} [2-Yoneda] Let  $\sfD\to \sfC$ be  a category fibered in
groupoids.  For any object $X\in \sfC$ here is an equivalence of
categories
\begin{align*}
\Theta: \Hom _\CFG (\underline{X}, \sfD)& \to  \sfD (X),\\
\quad (F:\underline{X}
\to \sfD )& \mapsto  F(X\stackrel{id}{\to} X),\\ \quad (\alpha:
F\Rightarrow G) &\mapsto  (\alpha (X\stackrel{id}{\to} X) : F
(X\stackrel{id}{\to} X) \to G (X\stackrel{id}{\to} X)).
\end{align*}
where $\Hom _\CFG (\underline{X}, \sfD)$ denotes the category of
maps of  CFGs and natural transformations between them.
\end{lemma}
\begin{proof}
Suppose $F,G: \underline{X} \to \sfD$ are two functors with $F(id_X) =
G(id_X) = \xi \in \sfD_0$.  We argue that for any $Y\in \sfC$ and any
$Y\stackrel{f}{\to} X\in \underline{X}(Y)_0$ there is a unique
$\alpha(f)\in C(Y)_1$ with $G(f)\stackrel{\alpha(f)}{\to} F(f)$.  Indeed, the diagram ${\vcenter{ \xy
(0,6)*+{Y}="1"; (0,-6)*+{X}="2";
(12,0)*+{X} ="3";
{\ar@{->}_{f} "1";"2"};
{\ar@{->}^{f} "1";"3"};
{\ar@{->}_{id} "2";"3"};
\endxy
}} $ in $\sfC$ defines an arrow in $\under{X}$ from
$(Y\stackrel{f}{\to}X)\in \under{X}(Y)_0$ to
$(X\stackrel{id}{\to}X)\in \under{X}(X)_0$.  Since $\pi_X
\left(\vcenter{ \xy (0,6)*+{Y}="1"; (0,-6)*+{X}="2"; (12,0)*+{X} ="3";
{\ar@{->}_{f} "1";"2"}; {\ar@{->}^{f} "1";"3"}; {\ar@{->}_{id}
"2";"3"};
\endxy}\right)
= (Y\stackrel{f}{\to} X) \in \sfC_1$ and
since $F$ and $G$ are maps of CFGs,
\[
\pi_\sfD (F\left(\vcenter{ \xy
(0,6)*+{Y}="1"; (0,-6)*+{X}="2";
(12,0)*+{X} ="3";
{\ar@{->}_{f} "1";"2"};
{\ar@{->}^{f} "1";"3"};
{\ar@{->}_{id} "2";"3"};
\endxy}\right))  = \pi_\sfD (G \left(\vcenter{ \xy
(0,6)*+{Y}="1"; (0,-6)*+{X}="2";
(12,0)*+{X} ="3";
{\ar@{->}_{f} "1";"2"};
{\ar@{->}^{f} "1";"3"};
{\ar@{->}_{id} "2";"3"};
\endxy}\right))  = Y\stackrel{f}{\to} X
 \]
as well.
Hence we have a diagram
\[
\xy
(0,8)*+{G(f)}="1"; (0,-8)*+{F(f)}="2";
(14,0)*+{\quad\quad \xi = G(id_X) = F(id_X)} ="3";
{\ar@{->}^{f} "1";"3"};
{\ar@{->}_{id} "2";"3"};
\endxy
\]
in $\sfD$. The functor $\pi_\sfD:\sfD \to \sfC$ takes the diagram
above to the commuting diagram
\[
 \xy
(0,6)*+{Y}="1"; (0,-6)*+{Y}="2";
(12,0)*+{X} ="3";
{\ar@{-->}_{id_Y} "1";"2"};
{\ar@{->}^{G(\triangleright)} "1";"3"};
{\ar@{->}_{F(\triangleright)} "2";"3"};
\endxy ,
\]
where
\[
\triangleright := \xy
(0,6)*+{Y}="1"; (0,-6)*+{X}="2"; (12,0)*+{X} ="3";
{\ar@{->}^f "1";"3"};{\ar@{->}_f "1";"2"}; {\ar@{->}^{id} "2";"3"};\endxy.
\]
Therefore, by the axioms of CFG, there is a unique arrow $\alpha(f)\in 
\sfD(Y)_1$ with $\pi_\sfD (\alpha(f)) = id_Y$ making the diagram
\[
\xy
(0,8)*+{G(f)}="1"; (0,-8)*+{F(f)}="2";
(14,0)*+{\xi} ="3";
{\ar@{->}^{G(\triangleright)} "1";"3"};
{\ar@{->}_{F(\triangleright)} "2";"3"};
{\ar@{->}_{\alpha(f)} "1";"2"};
\endxy
\]
commute.  The map $\alpha: \under{X}_0 \to \sfD_1$ is a natural transformation from $G$ to $F$.

We now argue that $\Theta$ is essentially surjective and fully
faithful.  Let $\xi \in \sfD(X)_0$ be an object.  Recall that for any
arrow $(Y\stackrel{f}{\to} X) \in \sfC_1$ we have chosen a pullback
$f^*\xi \in \sfD (Y)_0$.  Define a functor $F_\xi: \under{X}
\to \sfD$ by
\[
F_\xi (Y\stackrel{f}{\to} X) = f^*\xi,
\]

\[
F_\xi \left(\xy
(0,6)*+{Y}="1"; (0,-6)*+{Y'}="2";
(12,0)*+{X} ="3";
{\ar@{->}_{h} "1";"2"};
{\ar@{->}^{f} "1";"3"};
{\ar@{->}_{g} "2";"3"};
\endxy
\right) = \textrm{ the unique arrow in $\sfD$  from $f^*\xi$ to $g^*\xi$
covering $Y'\stackrel{h}{\to} Y$.}
\]
Note that $F_\xi (id_X) = id_X^*\xi = \xi$, so by the discussion above
there is a natural transformation $\alpha: F\Rightarrow F_\xi$.  Hence
$\Theta$ is essentially surjective.

It remains to prove that $\Theta$ is fully faithful.  Suppose
$(\gamma: \xi'\to \xi)\in \sfD(X)_1$ is an arrow.  We want to find a
natural transformation $\alpha_\gamma: F_{\xi'}\Rightarrow F_\xi$ with
$\Theta (\alpha_\gamma) = \gamma$ and prove that such a natural
transformation is unique.

Given $(Y\stackrel{f}{\to}X)\in \under{X}_0$ define
\[
\alpha_\gamma (Y\stackrel{f}{\to}X)
= (f^*\xi'\stackrel{f^*\gamma}{\to}f^*\xi).
\]
Then $\alpha_\gamma$ is a natural transformation from $F_{\xi'}$ to
$F_\xi$ with $\alpha_\gamma (id_X) = id_X^*\gamma = \gamma$.
Moreover $\alpha_\gamma$ is unique: if $\beta:\under{X}_0 \to \sfD_1$ is another natural transformation from $F_{\xi'}$ to
$F_\xi$ then for any $(Y\stackrel{f}{\to}X)\in \under{X}_0$ the diagram
\begin{equation}\label{eq**}
\xy
(-8,6)*+{f^*\xi'= F_{\xi'}(f)}="1"; (-8,-6)*+{f^*\xi = F_\xi (f)}="2";
(12,6)*+{\xi'} ="3";(12,-6)*+{\xi} ="4";
{\ar@{->}^{\beta(f)} "1";"2"};
{\ar@{->} "2";"4"};{\ar@{->} "1";"3"};{\ar@{->}_{\gamma} "3";"4"};
\endxy
\end{equation}
commutes in $\sfD$.  Since $\beta(f) \in \sfD(Y)_1$, $\pi_\sfD
(\beta(f)) = id_Y$.  Therefore $\pi_D$ takes the diagram (\ref{eq**})  to
\[
\xy
(0,6)*+{Y}="1"; (0,-6)*+{Y}="2";
(12,6)*+{X} ="3";(12,-6)*+{X} ="4";
{\ar@{->}_{id_Y} "1";"2"};
{\ar@{->} "2";"4"};{\ar@{->} "1";"3"};{\ar@{->}^{id_X} "3";"4"};
\endxy .
\]
By construction $\pi_\sfD$ also maps $\alpha_\gamma (f):f^*\xi' \to f^*\xi$ to $id_Y$ and makes
\[
\xy
(0,6)*+{f^*\xi'}="1"; (0,-6)*+{f^*\xi}="2";
(12,6)*+{\xi'} ="3";(12,-6)*+{\xi} ="4";
{\ar@{->}_{\alpha_\gamma (f)} "1";"2"};
{\ar@{->} "2";"4"};{\ar@{->} "1";"3"};{\ar@{->}^{\gamma} "3";"4"};
\endxy
\]
commute.  By (\ref{def:CFG}) we must have
\[
\alpha_\gamma (f) = \beta (f).
\]
Therefore $\Theta $ is fully faithful.
\end{proof}

\subsection{Atlases}

One last idea that we would like to describe in this fast introduction
to stacks is a way of determining a condition for a stack to be
isomorphic to a stack $BG$ for some Lie groupoid $G$.  This involves
the notion of an atlas, which, in turn, depends on a notion of a fiber
product of categories fibered in groupoids.
\begin{definition}
Let $\pi_X:X\to \sfC$, $\pi_Y:Y\to \sfC$ and $\pi_Z: Z\to \sfC$ be
three categories fibered in groupoids over a category $\sfC$.
The {\bf 2-fiber product} $Z\times_{X}Y$ of the diagram
  $\vcenter{\xymatrix@=12pt@M=10pt{
                     &  Y \ar[d]^f  \\
            Z  \ar[r]_(0.47)g        &     X      }}$
is  the category with objects
\[
(Z\times _X Y)_0=\left\{
              (y,z,\alpha ) \in Y_0 \times Z_0\times X_1 \mid
                                      \pi_{Y}(y)=\pi_{Z}(z), \,
                       f(y) \stackrel{\alpha}{\to} g(z)
                                 \right\}
\]
and morphisms
\begin{align*}
 &\Hom_{Z\times_{X}Y}\big( (z_1,y_1,\alpha_1),(z_2,y_2,\alpha_2) \big)=\\
&\left\{ (z_1\stackrel{v}{\to } z_2, y_1\stackrel{u}{\to} y_2)
\quad \left| \quad
                     \pi_{Y}(u)=\pi_{Z}(v) \in \sfC_1,
        \vcenter{ \xymatrix@R=8pt@C=10pt{
         f(y_1) \ar[r]^{f(u)}  \ar[d]_{\alpha_1} \ar@{}[rd]|{\circlearrowleft}
                                              &  f(y_2) \ar[d]^{\alpha_2} \\
                            g(z_1)  \ar[r]_{g(v)}  &  g(z_2)   }}\in X_1
\right . \right\}
\end{align*}
together with the functor $\pi: Z\times_{X}Y \to \sfC$ defined by
\[
\pi ((z, y,\alpha)) = \pi_Z (z) = \pi_Y (y), \quad
\pi (v,u) = \pi_Z (v) = \pi_Y (u)
\]
\end{definition}
\begin{remark}
It is not hard but tedious to check that $Z\times _X Y\to \sfC$ is a category
fibered in groupoids.
\end{remark}

\begin{remark}
There are two evident maps of CFGs $\pr_1: Z\times _X Y \to Z$ and $\pr_2:
Z\times _X Y \to Y$, but the diagram
$\vcenter{\xy
(0,6)*+{Z\times _X Y}="1"; (0,-6)*+{Z}="4";
(19,6)*+{Y} ="2";(19,-6)*+{X} ="3";
{\ar@{->}^{\pr_2} "1";"2"};
{\ar@{->}^{\pr_1} "1";"4"};
{\ar@{->}^{g} "2";"3"};{\ar@{->}_{f} "4";"3"};
\endxy}
$ does not strictly speaking commute.  Rather there is a natural
isomorphism $g\circ\pr_2 \Rightarrow f\circ \pr_1$ which need
not be the identity.
\end{remark}
\begin{remark} The fiber product $Z\times_{f,X,g}Y$ is characterized by the
following universal property: For any category fibered in groupoids $W$, 
there is a  natural equivalence of categories
\begin{align*}
\Hom(W, Z\times_X Y) &\to \\
&\{ (u,v,\alpha) \mid u:W\to Z, v:W\to Y 
\textrm{functors, } \alpha: u \Rightarrow v \textrm{ natural isomorphism }\};
\end{align*}
it sends a functor $h:W \to Z\times_X Y$ to the pair of functor
$h\circ \pr_1$, $h\circ \pr_2$ and the natural isomorphism between
them.
\end{remark}

\begin{example}
Let $G$ be a groupoid and $p:\underline{G_0} \to BG$ be the map of
CFGs defined by the canonical principal $G$-bundle $t:G_1 \to G_0$
($G$ acts on $G_1$ by multiplication on the right).  Then for any map
$f:\underline{M}\to BG$ from (the stack defined by) a manifold $M$ to
the stack $BG$, the fiber product $\underline{M}\times_{f, BG, p}
\underline{G_0}$ is (isomorphic to) $\underline{P_f}$, where $P_f\to M$
is the principal $G$-bundle corresponding to the map $f$ by 2-Yoneda.

\begin{proof}
We sort out what the objects of $\underline{M}\times_{f, BG, p}
\underline{G_0}$ are leaving the morphism as an exercise to the reader.
Fix a manifold $Y$.  The objects of the fiber $\underline{M}\times_{f,
BG, p} \underline{G_0} (Y)$ are triples $(z,y, \alpha)$ where $z\in
\under{M}(Y)_0$, $y\in \under{G_0}(Y)$ and $\alpha$ is an arrow in
$BG(Y)$ from $f(z)$ to $p(y)$.  The objects of $\under{M}(Y)$ are maps
of manifolds $Y\stackrel{k}{\to}M$.  The image
$f(Y\stackrel{k}{\to}M)$ of such an object is a principal $G$-bundle
over $Y$.  By 2-Yoneda this bundle is $k^*P_f$ (recall that $P_f =
f(id_M)\in BG (M)$).  Similarly $p(Y\stackrel{\ell}{\to}G_0) = \ell^*
(G_1\to G_0)$.  Finally $\alpha: f(Y\stackrel{k}{\to}M) \to
p(Y\stackrel{\ell}{\to}G_0)$ is an arrow in the category $BG(Y)$.
That is, $\alpha: k^*P_f \to \ell^* (G_1\to G_0)$ is an isomorphism of
two principal $G$-bundles over $Y$.  Note that since $G_1\to G_0$ has
a global section, the pullback $\ell^*(G_1\to G_0)$ also has a global
section.  And the isomorphism $\alpha\inv: \ell^*(G_1\to G_0) \to
k^*P_f$ is uniquely determined by the image of this global section.
Hence the objects of $\underline{M}\times_{f, BG, p} \underline{G_0}
(Y)$ are pairs (pullback to $Y$ of $P_f\to M$, global section of the
pullback).  A global section of $k^*P_f \to Y $ uniquely determines a
map $\sigma :Y\to P_f$ making the diagram
\[
\xy
(0,0)*+{Y}="1"; (12,12)*+{P_f}="2";
(12,0)*+{M} ="3";
{\ar@{->}^{\sigma } "1";"2"};
{\ar@{->}_{k} "1";"3"};{\ar@{->} "2";"3"};
\endxy
\]
commute.  Therefore objects of $\underline{M}\times_{f, BG, p} \underline{G_0}
(Y)$ ``are'' maps from $Y$ to $P_f$.

Unpacking the definitions further one sees that
$\underline{M}\times_{f, BG, p} \underline{G_0}$ is isomorphic to
$\under{P_f}$ as a category fibered in groupoids, where by
``isomorphic'' we mean ``equivalent as a category.''
\end{proof}
\end{example}

\begin{remark}
The map of manifolds $P_f \to M$ in the construction above is a
surjective submersion.  Therefore we may think of $\under{G_0}
\stackrel{p}{\to} BG$ as a surjective submersion.
\end{remark}

\begin{remark}
To keep the notation from getting out of control we now drop the
distinction between a manifold $M$ and the associated stack
$\under{M}$.  We will also drop the distinction between stacks
isomorphic to manifolds and manifolds.  Thus, in the example above we
would say that for any Lie groupoid $G$, any manifold $M$ and any map
$M\to BG$ the fiber product $M\times _{BG} G_0$ is a manifold.
\end{remark}

\begin{definition}[Atlas of a stack]
Let $\sfD\to \Man$ be a stack over the category manifolds.  An {\em
atlas} for $\sfD$ is a manifold $X$ and a map $p:X\to \sfD$ such that
for any map $f:M\to \sfD$ from a manifold $M$ the fiber product
$M\times _{f,D,p} X$ is a manifold and the map $\pr_1:M\times _{f,D,p}
X \to M$ is a surjective submersion.
\end{definition}

\begin{remark}
A stack over manifolds which possesses an atlas is alternatively
refered to as a {\em geometric } stack, {\em differentiable } stack or
an {\em Artin} stack.
\end{remark}
\begin{example}
Let $M$ be a manifold $\cU = \bigsqcup U_i \to M$ be a cover by
coordinate charts.  Then the map of stacks $p: \under{\cU} \to
\under{M}$ is an atlas.
\end{example}
\begin{example}
For any Lie groupoid $G$ the canonical map $p:G_0 \to
BG$ sending $id_{G_0}$ to the principal $G$-bundle $G_1 \to G_0$ is an
atlas.
\end{example}

\begin{proposition} \label{prop:bg-from-atlas}
Given a stack with an atlas $p:X\to \sfD$ there is a Lie groupoid $G$
such that $\sfD$ is isomorphic to $BG$.  Moreover we may take $G_0 =
X$ and $G_1 = X\times_{p, \sfD,p} X$.
In other words any geometric  stack $\sfD$ is $BG$ for some Lie groupoid $G$.
\end{proposition}

It is relatively easy to produce the groupoid $G$ out of the atlas
$p:X\to \sfD$.  It is more technical to define a map of stacks
$\psi:\sfD \to BG $ and to prove that it is an isomorphism of stacks
(that is, prove that $\psi$ is an equivalence of categories commuting
the projections $\pi_{BG}:BG \to \Man$ and $\pi_\sfD:\sfD \to \Man$).
We will only sketch its construction and refer the reader to stacks
literature for a detailed proof.  The reader may consult, for example,
\cite{Metzler}[Proposition~70].

\begin{proof}[Sketch of proof of Proposition~\ref{prop:bg-from-atlas}]
We first construct a Lie groupoid out of an atlas on a stack.  Let
$\sfD$ be a stack over manifolds and $p:{G_0}\to \sfD$ an atlas.
Then the stack ${G_0}\times _{p, \sfD, p} {G_0}$ is a
manifold; call it $G_1$.  We want to produce the five structure maps:
source, target $s,t :G_1 \to G_0$, unit $u: G_0 \to G_1$, inverse
$i:G_1\to G_1$ and multiplication $m: G_1 \times _{G_0} G_1 \to G_1$
satisfying the appropriate identities.  We will produce five maps of
stacks.  By Corollary~\ref{cor:maps-of-manifolds-as-stacks} this is
enough.  We take as source and target the projection maps $\pr_1,
\pr_2:
{G_0}\times _{p, \sfD, p} {G_0} \to {G_0}$.  
Since the diagram 
\[
\xy (0,6)*+{{G_0}}="1"; (0,-6)*+{{G_0}}="2"; 
(12,6)*+{{G_0}} ="3";(12,-6)*+{\sfD} ="4"; {\ar@{->}_{id}
"1";"2"}; {\ar@{->}^{id} "1";"3"}; {\ar@{->}_{p}
"2";"4"};{\ar@{->}_{p} "3";"4"};
\endxy
\]
commutes, there is a unique map of stacks $u: {G_0} \to
{G_0}\times _{p, \sfD, p} {G_0}$.  Concretely, on objects,
it sends $x\in {G_0}$ to $(x,x, id_{p(x)})$.  We also have the
multiplication functor
\[
m: ({G_0}\times _{p, \sfD, p} {G_0})\times _{{G_0}}
({G_0}\times _{p, \sfD, p} {G_0} ) \to 
({G_0}\times _{p, \sfD, p} {G_0}),
\]
which on objects is given by composition:
\[
 m ((x_1, x_2, \alpha), (x_2, x_3, \beta)) = (x_1, x_3, \beta \alpha).
\]
It is easy to see that the multiplication is associative.
Finally the inverse map 
\[
inv :{G_0}\times _{p, \sfD, p} {G_0} \to 
{G_0}\times _{p, \sfD, p} {G_0}
\]
is given, on objects, by
\[
inv (x_1, x_2, \alpha) = (x_2, x_1, \alpha\inv).
\]

Note that the construction above does not use the descent properties
of $\sfD$.  That is, we could have just as well defined an atlas for a
category fibered in groupoids.  The construction would then still
produce a Lie groupoid.

Next we sketch a construction of a map $\psi: \sfD \to BG$ of CFGs.
It will turn out to be a fully faithful functor.  We will only need
the fact that $\sfD$ is a stack to prove that $\psi$ is essentially
surjective.

By 2-Yoneda, an object of $\sfD$ over a manifold $M$ is a map of CFGs
$f:M\to \sfD$.  Since $p:X\to \sfD$ is an atlas, the fiber product
$M\times _\sfD X$ is a manifold and the map $\pr_1:M\times _\sfD X \to
M$ is a surjective submersion.  There is a free and transitive action
of $G$ on the fibers of $\pr_1$ with respect to the anchor map $\pr_2
M\times _\sfD X \to X = G_0$ (once again we identify manifolds with
the corresponding stacks).  The right action of $G$ is given by the
``composition''
\begin{align*}
(M\times _\sfD X)\times _X (X \times _\sfD X)& \to M\times _\sfD X\\
((x_1, x_2, \alpha), (x_2, x_3, \beta ) &\mapsto (x_1, x_3, \beta \alpha)
\end{align*}
(following the tradition in the subject we only wrote out the map on objects).
It is free and transitive since the map
\begin{align*}
(M\times _\sfD X)\times _X (X \times _\sfD X)& \to (M\times _\sfD X) \times _M 
(M\times _\sfD X) \\
((x_1, x_2, \alpha), (x_2, x_3, \beta )) &\mapsto ((x_1,  x_2, \alpha), 
(x_2, x_3, \beta \alpha))
\end{align*}
is an isomorphism of stacks.  Thus 
\[
 \psi (f:M\to D) = (\pr_1: M\times _{f, \sfD, p} X \to M).
\]
Next we define $\psi$ on arrows.  An arrow from $f_1:M_1 \to \sfD$ to
$f_2: M_2 \to \sfD$ is a 2-commuting triangle $\vcenter{\xy
(-8,8)*+{M_1}="1"; (-8,-8)*+{M_2}="2"; (8,0)*+{\sfD}="3";
{\ar@{->}_{h} "1";"2"}; {\ar@{->}^{f_1} "1";"3"};{\ar@{->}_{f_2}
"2";"3"};
 {\ar@{=>}^<<<{} (-3,-2)*{};(0,1)*{}} ;
\endxy }$ (this can be proved more or less the same way as we proved 2-
Yoneda).  Since the diagram
$\vcenter{
\xy
(-8,8)*+{M_1 \times _\sfD X }="1"; (-8,-8)*+{M_1}="2"; (8,8)*+{X}
="3"; (8,-8)*+{\sfD}="4"; {\ar@{->}_{h\circ \pr_1 } "1";"2"};
{\ar@{->}^{} "1";"3"};{\ar@{->}_{f_1} "2";"4"}; {\ar@{->}^p "3";"4"}
;{\ar@{=>}^<<<{} (-2,-2)*{};(1,1)*{}} ;
\endxy}$
2-commutes, we get, by the universal property of the 2-fiber product, a map
\[
\tilde{h}: M_1 \times _\sfD X \to M_2 \times _\sfD X 
\]
making the diagram
\[
\xy
(-12,8)*+{M_1 \times _\sfD X }="1"; (-12,-8)*+{M_1}="2"; 
(12,8)*+{M_2 \times _\sfD X} ="3";
(12,-8)*+{M_2}="4";
{\ar@{->}_{ } "1";"2"}; {\ar@{->}^{\tilde{h}} "1";"3"};{\ar@{->}_{h}
"2";"4"}; {\ar@{->} "3";"4"} ;
\endxy
\]
2-commute. And since all the objects in the diagram are manifolds, it
actually commutes on the nose.  It is not hard to check that
$\tilde{h}$ is compatible with the action of $G$.  This defines $\psi$
on arrows and gives us a functor
\[
\psi : \sfD \to BG.
\]
One checks that $\psi$ is fully faithful (I am waving my hands here).

Next we argue that the full subcategory $BG\triv$ of $BG$ consisting
of the trivial bundles is in the image of $\psi$.  A trivial $G$
bundle on a manifold $M$ is the pull back of the unit $G$-bundle $G_1
\to G_0 = X$ by a map $k:M\to X$.  The diagram $\vcenter{\xy
(-8,0)*+{M}="1"; (8,8)*+{X}="2"; (8,0)*+{\sfD}="3"; {\ar@{->}^{k}
"1";"2"}; {\ar@{->}_{p\circ k} "1";"3"};{\ar@{->}^{p} "2";"3"}; \endxy
}$ commutes by definition.  Hence $\pr_1: M\times _\sfD X \to M$ has a
global section $\sigma $ with $\pr_2 \circ \sigma =k$.  Therefore
$M\times _\sfD X \to M$ is isomorphic to $k^* (G_1 \to G_0) \to M$.
That is,
\[
   \psi (p\circ k) \simeq k^*(G_1\to G_0).  
\]
Similarly if 
\[
\xy (-8,8)*+{M_1}="1"; (-8,-8)*+{M_2}="2";
(8,0)*+{X}="3"; {\ar@{->}_{h} "1";"2"}; {\ar@{->}^{ k_1}
"1";"3"};{\ar@{->}_{ k_2} "2";"3"}; 
\endxy
\]
is a commuting diagram of maps of manifolds, then 
\[
\xy (-8,8)*+{M_1}="1"; (-8,-8)*+{M_2}="2";
(8,0)*+{\sfD}="3"; {\ar@{->}_{h} "1";"2"}; {\ar@{->}^{p\circ k_1}
"1";"3"};{\ar@{->}_{p\circ k_2} "2";"3"}; 
\endxy
\]
is a commuting triangle of maps of CFGs, i.e., a map between two objects in $\sfD$.  One checks that $\psi (\vcenter{
\xy (-8,8)*+{M_1}="1"; (-8,-8)*+{M_2}="2";
(8,0)*+{\sfD}="3"; {\ar@{->}_{h} "1";"2"}; {\ar@{->}
"1";"3"};{\ar@{->} "2";"3"}; {\ar@{=>}^<<<{} (-3,-2)*{};(0,1)*{}} ;
\endxy})$ is the map $\tilde{h} : k_1^* (G_1 \to G_0) \to k_2^* (G_1 \to G_0)$.
Thus the image of $\psi $ includes the full subcategory $BG\triv$ of
trivial bundles.

Finally we use the fact that $\sfD$ is a stack to argue that $\psi$ is
essentially surjective.  If $P\to M$ is a principal $G$-bundle, then
$M$ has an open cover $\{U_i\to M\}$ so that the restrictions
$P|_{U_i}$ have global sections.  Then for each $i$ there is $\xi_i
\in D(U_i)_0$ with $\psi (\xi_i)$ isomorphic to $P|_{U_i}$.  The cover
also defines descent data $(\{P|_{U_i}\}, \{ \phi_{ij}\})$.  These
descent data really live in $BG\triv$.  Hence, since the image of
$\psi$ contains $BG\triv$ and since $\psi$ is fully faithful,
$(\{P|_{U_i}\}, \{ \phi_{ij}\})$ defines descent data $(\{ \xi_i\},
\{\psi\inv( \phi_{ij})\})$ in $\sfD$.  Since $\sfD$ is a stack, these
descent data defines an object $\xi$ of $\sfD (M)$.  Since $\psi$ is a
functor, $\psi (\xi)$ is isomorphic to $P$.  We conclude that
$\psi:\sfD \to BG$ is essentially surjective.
\end{proof}

\begin{remark}
Atlases of geometric stacks are not unique.  For example, if $p:X\to
\sfD$ is an atlas and $f:Y\to X$ is map of manifolds which  is a
surjective submersion, then $p\circ f:Y\to \sfD$ is also an atlas.
However, if $p:{G_0} \to \sfD$ and $q: {H_0}\to \sfD$ are two atlases,
then by Proposition~\ref{prop:bg-from-atlas}, the stacks $BG$ and $BH$
are isomorphic.  It is not hard to construct an invertible bibundle
$P:G\to H$ explicitly: $P$ is the fiber product $G_0 \times_{, \sfD,
q} H_0$.  The actions of $G$ and $H$ are defined as in the proof of
Proposition~\ref{prop:bg-from-atlas} and they are both principal.  

It is useful to think of these two atlases and of the two
corresponding Lie groupoids as two choices of ``coordinates'' on the
stack $\sfD$.
\end{remark}

\begin{remark}
In the light of the above remark it makes sense to say that {\bf a
geometric stack} $\sfD\to \Man$ is {\bf an orbifold} if there is an
atlas $p:X\to \sfD$ so that the corresponding groupoid $X\times _\sfD
X \toto X$ is a proper etale Lie groupoid.
\end{remark}

\end{document}